\documentclass[11 pt,fullpage]{article}
\usepackage{amsfonts,amssymb}
\usepackage{graphicx}
\usepackage{amsmath,amsthm}
\usepackage{color} 
\usepackage{hyperref}
\usepackage{caption}
\usepackage[margin=2.5cm]{geometry}

\usepackage{bmpsize}

\usepackage{mathtools,enumitem,mathrsfs}
\mathtoolsset{showonlyrefs=true}

\theoremstyle{plain}
\newtheorem{theorem}{Theorem}

\newtheorem{proposition}{Proposition}[section]
\newtheorem{lemma}[proposition]{Lemma}

\newtheorem{definition}{Definition}[section]

\theoremstyle{definition}
\newtheorem{remark}{Remark}

\allowdisplaybreaks

\numberwithin{equation}{section}

\newcommand\N{{\mathbb N}}
\newcommand\R{{\mathbb R}}
\newcommand\T{{\mathbb T}}

\newcommand\C{{\mathbb C}}

\newcommand{\cE}{\mathcal E}

\newcommand{\cL}{\mathcal L}

\newcommand{\cO}{\mathcal O}

\newcommand{\cR}{\mathcal R}

\def\eps{{\varepsilon}}

\renewcommand\eps{\epsilon}

\newcommand{\Real}{\mathbb R}
\newcommand{\Complex}{\mathbb C}

\newcommand{\norm}[1]{\left\lVert#1\right\rVert}
\newcommand{\abs}[1]{\left\vert#1\right\vert}
\newcommand{\set}[1]{\left\{#1\right\}}

\newcommand{\grad}{\nabla}

\newcommand{\brak}[1]{\langle #1 \rangle}

\newcommand{\dss}{\displaystyle}

\newcommand{\var}{\varepsilon}

\begin{document}
 
\title{Linearized wave-damping structure of Vlasov-Poisson in
  $\Real^3$} \author{Jacob
  Bedrossian\thanks{\footnotesize Department of Mathematics, University of Maryland, College Park, MD 20742, USA \href{mailto:jacob@math.umd.edu}{\texttt{jacob@math.umd.edu}}. J.B. was supported by NSF CAREER grant DMS-1552826 and NSF RNMS \#1107444 (Ki-Net)} \,
  Nader Masmoudi\thanks{\footnotesize NYUAD Research Institute, New York University Abu Dhabi, PO Box 129188, Abu Dhabi,   United Arab Emirates.
    Courant Institute of Mathematical Sciences, New York University, 251 Mercer Street, New York, NY 10012, USA. \textit{masmoudi@cims.nyu.edu}. The work of N. M is supported by  NSF grant DMS-1716466 and  by Tamkeen under the NYU Abu Dhabi Research Institute grant
of the center SITE.}
  \, and Cl\'ement Mouhot\thanks{\footnotesize Centre for Mathematical Sciences, University of Cambridge.  \textit{c.mouhot@dpmms.cam.ac.uk} Partially funded by ERC grant MAFRAN.}  }

\date{\today}
\maketitle

\begin{abstract}
  In this paper we study the linearized Vlasov-Poisson equation for
  localized disturbances of an infinite, homogeneous Maxwellian
  background distribution in $\Real^3_x \times \Real^3_v$.  In
  contrast with the confined case $\T^d _x \times \R_v ^d$, or the
  unconfined case $\R^d_x \times \R^d_v$ with screening, the dynamics
  of the disturbance are not scattering towards free transport as
  $t \to \pm \infty$: we show that the electric field decomposes into
  a very weakly-damped Klein-Gordon-type evolution for long waves and
  a Landau-damped evolution. The Klein-Gordon-type waves solve, to
  leading order, the compressible Euler-Poisson equations linearized
  about a constant density state, despite the fact that our model is
  collisionless, i.e. there is no trend to local or global
  thermalization of the distribution function in strong topologies. We
  prove dispersive estimates on the Klein-Gordon part of the
  dynamics. The Landau damping part of the electric field decays
  faster than free transport at low frequencies and damps as in the
  confined case at high frequencies; in fact, it decays at the same
  rate as in the screened case. As such, neither contribution to the
  electric field behaves as in the vacuum case. 
\end{abstract}


\setcounter{tocdepth}{2} 
{\small\tableofcontents}

\section{Introduction}\label{sec:Intro}
\subsection{The problem at hand}


One of the fundamental equations in the kinetic theory of plasmas is
the Vlasov-Poisson equations for an infinitely extended plasma (see
e.g. \cite{GoldstonRutherford95,BoydSanderson}),
\begin{equation}
  \label{def:VPE}
  \left\{
    \begin{array}{l} \dss 
      \partial_t f + v\cdot \grad_x f + E(t,x)\cdot \grad_v f = 0, \\[2mm]
      \dss E(t,x) = -\grad_x W \ast_{x} \rho(t,x), \\
      \dss \rho(t,x) = \int_{\R^3} f(t,x,v) dv - n_0, \\[3mm]
      f(t=0,x,v) = f_{in}(x,v),
    \end{array}
  \right.
\end{equation}
for the time-dependent probability density function $f(t,x,v) \ge 0$
of the electrons in the phase space $(x,v) \in \R^3 \times \R^3$,
$n_0$ is the number density of the constant ion background, and where
$W$ is the kernel of the Coulomb interaction\footnote{Note that we are slightly abusing notation as $E$ in \eqref{def:VPE} actually denotes the acceleration of the electrons due to the electric field, not the electric field itself. This convention is taken for the remainder of the paper.}.
\begin{align}
  \label{def:Coulomb}
  W(x) = \frac{q^2}{4 \pi \epsilon_0 m_e \abs{x}}
\end{align}
with $q$ the electron charge, $m_e$ the electron mass, and
$\epsilon_0$ the vacuum permittivity. 
We will consider \eqref{def:VPE} linearized about the homogeneous Maxwellian background
with fixed temperature $T$
\begin{align}
  \label{eq:f0}
  f^0(v) := \frac{n_0}{(2\pi T)^{3/2}} e^{-\frac{m_e \abs{v}^2}{2T}}. 
\end{align}
In 1930s and 1940s, Vlasov~\cite{Vlasov-damping,Vlasov1945} suggested
to neglect collisions and derive the so-called Vlasov-Poisson equation
for long-range interactions, independently from Jeans'
derivation~\cite{jeans1915theory} in stellar dynamics. Motivated by
the mathematical understanding of the plasma oscillations previously
theorized in particular by Langmuir, he studied the linearized
approximation of~\eqref{def:VPE} (see~\eqref{def:VPElin} below)
formally searching for eigenmodes in the form of planar waves
$e^{-i\omega t + ikx}F(v)$, given some velocity distribution $F$, and
computed various dispersion relations (see for instance equation~(50)
in~\cite{Vlasov-damping}). He asserted, not quite correctly but
\emph{almost correctly} as we shall see and clarify in this paper,
that in the long-wave limit $\abs{k} \ll 1$, where $k$ is the Fourier
variable in space, one has the dispersion relation
\begin{align}
  \label{eq:BG}
  \omega^2 = \omega_p^2 +  \frac{3T}{m_e} \abs{k}^2 +
  \cO\left(\abs{k}^4\right)
  \quad \textup{ as } k \to 0, \quad \text{ where } \quad
  \omega_p := \frac{q^2 n_0}{\epsilon_0 m_e}
\end{align}
is the \emph{cold plasma frequency}. The relation
$\omega^2 = \omega_p^2 + \frac{3T}{m_e} \abs{k}^2$ is called the
\emph{Bohm-Gross dispersion relation}\footnote{Technically, the
  relation seems to have first appeared in Vlasov's earlier
  work~\cite{Vlasov-damping} before appearing in \cite{BG49A},
  however, it is likely that access to Vlasov's work was difficult at
  that time. See also e.g. pg 260 of \cite{GoldstonRutherford95}.},
and arises from the following compressible Euler-Poisson for the
spatial density $n(t,x) \ge 0$ and macroscopic velocity
$u(t,x) \in \R^3$
\begin{align*}
  \left\{
    \begin{array}{l} \dss 
      \partial_t n + \nabla_x \cdot \left( n  u \right) = 0 \\[2mm]
      \dss m_e n \left[ \partial_t u + (u \cdot \nabla_x)u \right] = q
      n E(t,x) - 3 T \grad_x n \\[2mm]
      \dss \epsilon_0 \grad_x \cdot E(t,x) = q n
    \end{array}
  \right.
\end{align*}
when linearized around a constant density state $n_0 >0$, $u_0 =0$:
\begin{align}
  \label{eq:lin-Euler}
  \left\{
    \begin{array}{l} \dss 
      \partial_t n + n_0 \grad_x \cdot u = 0 \\[2mm]
      \dss m_e n_0 \partial_t u = q n_0 E(t,x) - 3T\grad_x n \\[2mm]
      \dss \epsilon_0 \grad_x \cdot E(t,x) = q n.
    \end{array}
  \right.
\end{align}
This model is sometimes referred to as a \emph{warm plasma model} in
the physics literature see e.g. [Chapter 16;
\cite{GoldstonRutherford95}] for more detail. Vlasov's prediction was
shown to be incomplete by Landau in 1946 \cite{Landau46}, who showed
that in fact the linearized electric field decays to zero as
$t\to\infty$, a phenomenon now known as \emph{Landau damping}, and in
particular that non-trivial planar waves never satisfy the
dynamics. The damping arises due to \emph{mixing/filamentation} in phase
space, not unlike a scalar quantity being stirred in a fluid (as first
pointed out in \cite{VKampen55}). Landau damping was later observed in
experiments~\cite{MalmbergWharton64,MalmbergWharton68} and is
considered one of the most important properties of collisionless
plasmas~\cite{Ryutov99,GoldstonRutherford95,Swanson}.

On $\mathbb T^3 \times \mathbb R^3$, the linearized electric field
decays rapidly provided the initial data is regular, specifically, its
decay in time is comparable to the decay in Fourier variable
associated with $v$ of the initial data~\cite{MouhotVillani11,BMM13}.
Analogous estimates are also true for the density on
$\mathbb R^3 \times \mathbb R^3$ if one uses a model with Debye
shielding (this arises when studying the ion distribution function),
i.e. where $W$ is the fundamental solution to the elliptic problem
$-\Delta W + \alpha W = \delta$ for some constant $\alpha> 0$ (see
\cite{BMM16,HKNR19} and the references therein). However, such rapid
decay estimates on the linearized electric field are false in
$\mathbb R^3 \times \mathbb R^3$ for Vlasov-Poisson. Landau himself
predicted an extremely slow Landau damping of long waves approximately
solving the dispersion relation \eqref{eq:BG} (see also
e.g. \cite{GoldstonRutherford95,Swanson} for modern exposition in the
physics literature), something that we will make much more precise
below.  This lack of significant Landau damping was proved rigorously
by Glassey and Schaefer~\cite{Glassey94,Glassey95}, though a precise
characterization of the dynamics was still lacking.

In this work, we precise the linearized dynamics and prove that the
linearized electric field can be split into two contributions: one
contribution at long spatial waves $k \sim 0$ that is a
Klein-Gordon-like propagation matching \eqref{eq:BG} to leading order
with very weak Landau damping with rate ``$\mathcal{O}(|k|^\infty)$''
and another contribution, which is properly Landau damping and decays
at a rate \emph{faster} than kinetic free transport for long-waves
$k \sim 0$, i.e. faster than Vlasov-Poisson equation linearized around
$f^0 = 0$, $n_0=0$.
In particular, our work shows that the hydrodynamic description
(linearized Euler-Poisson) in fact is the leading order description of
the electric field at long waves, despite the lack of collisional
effects. As such, we remark that Vlasov was essentially correct to
leading order in his prediction of the long wave dynamics of
\eqref{def:VPElin}.  For long waves, we show that the distribution
function decomposes (to leading order) into two pieces: a term that
factorizes as $\tilde{E}(t,x)\cdot \grad_v f^0(v)$ (where $f^0$ is the
Maxwellian background) where $\tilde{E}$ solves a Klein-Gordon-type
equation (including an additional tiny damping) and a separate
contribution that scatters to free transport in $L^p_{x,v}$ for
$p > 6$.  Hence, the hydrodynamic behavior arises from a large-scale
collective motion of the plasma that is insensitive to the
filamentation in phase space normally associated with Landau damping.

Understanding the relationship between Landau damping (or other
kinetic effects) and the observed large-scale hydrodynamic behavior in
collisionless plasmas has been an area of research in the physics
literature for some time (see
e.g. \cite{HP90,HDP92,SHD97,HunanaEtAl18} and the references therein).
Our work provides a precise description of the hydrodynamic behavior
and its leading order corrections due to Landau damping and other kinetic effects, for the simple linearized problem; studies of more
physical settings (e.g. external and/or self-consistent magnetic
fields, inclusion of ions, inhomogeneous backgrounds etc) and/or the
inclusion of nonlinear effects is an interesting direction of further
research.

Our work is only linear, but we remark that much progress has been
made at the nonlinear level in recent years in other settings.  After
the earlier work of \cite{CagliotiMaffei98} (see
also~\cite{HwangVelazquez09}), the major breakthrough came in
\cite{MouhotVillani11}, when Landau damping for the nonlinear problem
was shown on $\mathbb T^3 \times \mathbb R^3$ for all sufficiently
small and smooth initial perturbation of Landau-stable stationary
solutions; the actual smoothness required being Gevrey or analytic. We
also refer to~\cite{BMM13,GNR20} for simplified proofs and~\cite{B17}
for a study of the problem with collisions. The work~\cite{B16} shows
that the results therein do not hold in finite regularity (see also
\cite{GNR20II}). Our previous work \cite{BMM16} studied the nonlinear
problem with shielding on $\mathbb R^3 \times \mathbb R^3$; see also
\cite{HKNR19,N20} for an alternative approach and some
refinements. Other works, old and new, have studied the nonlinear
Cauchy problem near the vacuum state, that is without the presence of
a non-zero spatially homogeneous background equilibrium $f^0$, see
e.g. \cite{BardosDegond1985,IPWW20}.  As can be seen from our results
below, the dynamics are significantly different in this case.

\subsection{Main results}

We linearize the Vlasov-Poisson equation~\eqref{def:VPE} on
$\Real^3_x \times \Real^3_v$ around an infinitely-extended,
homogeneous background $f^0(v) \ge 0$. This models a spatially
localized disturbance of an infinite plasma in which collisions can be
neglected, a fundamental problem in the kinetic theory of
plasmas~\cite{GoldstonRutherford95,BoydSanderson,Swanson}.
For a localized disturbance $h := f^0 -f$, the \emph{linearized}
Vlasov-Poisson equations for the electron distribution function are
given by
\begin{equation}
  \label{def:VPElin}
  \left\{
    \begin{array}{l} \dss 
      \partial_t h + v\cdot \grad_x h + E(t,x)\cdot \grad_v f^0 = 0, \\[2mm]
      \dss E(t,x) = -\grad_x W \ast_{x} \rho(t,x), \\
      \dss \rho(t,x) = \int_{\R^3} h  (t,x,v) dv, \\[3mm]
      \dss h(t=0,x,v) = h_{in}(x,v),
    \end{array}
  \right.
\end{equation}
where we take the charge neutrality assumption
$\int_{\R^3 \times \R^3} h_{in}(x,v) dx dv = 0$.
Define the standard plasma constants (number density, plasma frequency
and temperature):
\begin{align*}
  n_0 := \widehat{f^0}(0), \quad
  \omega_p^2 := \frac{q^2 n_0}{\epsilon_0 m_e}, \quad 
  T := \frac{m_e}{3n_0} \int_{\mathbb R^3} \abs{v}^2 f^0(v) dv, 
\end{align*}
and as stated above, we assume the Maxwellian distribution background
\eqref{eq:f0}.

Our main results are asymptotic decomposition of the electric field
and the distribution function; we give slightly simplified statements
for readability and refer to the main body of the paper for more
detailed expansions. See Subsection~\ref{subsec:notation} for the
notation $\langle x \rangle$, $\langle x,y \rangle$, $\langle
\nabla\rangle$, and $W^{k+0,p}_{w}$.

We remark that another preprint proving similar results with somewhat
different methods \cite{HKNF20} has recently been completed as well.
These works have been completed totally independently.

\begin{theorem} \label{thm:ElecDec} Suppose
  $\int_{\R^3 \times \R^3} h_{in} dx dv = 0$ and let $h$ and $E$ solve
  \eqref{def:VPElin}.  There is $\delta_0>0$ and a decomposition of
  the electric field $E = E_{KG} + E_{LD}$ between a `Klein-Gordon'
  and `Landau damped' parts, with $E_{KG}$ supported in spatial
  frequencies $|k|< \nu_0$ and:
  \begin{itemize}
  \item $E_{LD}$ satisfies the following \textbf{Landau-damping-type
      decay estimates} for any $\sigma \in \N$:
    \begin{align}
      \label{ineq:EFTbdL2}
      \norm{\brak{\grad_x, t\grad_x}^\sigma E_{LD}(t)}_{L^2_x}
      & \lesssim \frac{1}{\brak{t}^{5/2}}
        \norm{h_{in}}_{W^{\sigma+ \frac{3}{2} +0,1}_{0}}  \\
      \label{ineq:EFTbd}
      \norm{\brak{\grad_x, t\grad_x}^\sigma E_{LD}(t)}_{L^\infty_x}
      & \lesssim \frac{1}{\brak{t}^{4}} \norm{h_{in}}_{W^{\sigma+3+0,1}_{0}}. 
    \end{align}
  \item $E_{KG}$ further decomposes as $E_{KG} = E_{KG}^{(1)} +
    E_{KG}^{(2)}$ with the \textbf{pointwise-in-time estimates}:
    \begin{align*}
      & \norm{E_{KG}^{(1)}(t)}_{L^2_x}
        \lesssim \norm{E_{in}}_{L^2} + \norm{v h_{in}}_{L^2_{x,v}} +
        \norm{h_{in}}_{W^{0,1}_{4}} \\
      & \norm{E_{KG}^{(2)}(t)}_{L^2_x} \lesssim
        \norm{h_{in}}_{W^{0,1}_{5}} \\
      & \norm{E_{KG}^{(2)}(t)}_{L^{\infty}_x}
        \lesssim \brak{t}^{-3/2} \norm{h_{in}}_{W^{3/2,1}_{5}} \\
      & \norm{E_{KG}^{(1)}(t)}_{L^{p}_x} \lesssim_p
        \brak{t}^{-3\left(\frac12 - \frac{1}{p}\right)}
        \norm{\brak{x} h_{in}}_{W^{0,1}_{4}} \quad (\forall \, 2 \leq p < \infty) \\ 
      & \norm{\grad_x E_{KG}^{(1)}(t)}_{L^{\infty}_x} \lesssim
        \brak{t}^{-3/2} \norm{h_{in}}_{W^{0,1}_{4}}.
    \end{align*}
  \item
    $E_{KG}^{(1)}$ solves a weakly damped Klein-Gordon type equation
    in the following sense: there are bounded, smooth functions
    $\lambda, \Omega$, $k \in
    B(0,\delta_0)$ such that (for errors independent of $t$),
    \begin{align*}
      & \widehat{E_{KG}^{(1)}}(t,k) = \widehat{E}_{in}(k)
        e^{-\lambda(k) t}
        \cos \left( \Omega(k) t \right)
        + e^{-\lambda(k) t} \frac{ik}{\abs{k}^2} \left( k \cdot \grad_\eta
        \widehat{h_{in}}(k,0) \right)
         \frac{\sin \left(\Omega(k) t
        \right)}{\Omega(k)}  \\
      & \qquad\qquad\qquad + \cO(\abs{k}^2)e^{-\lambda(k) t+
        i\Omega(k) t}
        + \cO(\abs{k}^2)e^{-\lambda(k) t - i\Omega(k) t}, \\
    & \text{where } \quad \lambda(k) >0, \quad \lambda(k) = \cO\left( \abs{k}^\infty
      \right), \quad  \Omega^2(k)  = \omega^2_p + \frac{3T}{m_e}\abs{k}^2 + \cO\left(\abs{k}^4 \right)
      \quad \text{ as } \quad k \to 0.
    \end{align*}
    Furthermore, here holds for all $2 < p \leq \infty$, 
    \begin{align*}
      \norm{(-\Delta_x)^{-1}E_{KG}^{(2)}(t)}_{L^p}
      \lesssim
      \brak{t}^{-3\left(\frac{1}{2}-\frac{1}{p} \right)} \norm{h_{in}}_{W^{3/2,1}_{5}}, 
    \end{align*}
    which further emphasizes that for small $k$, $E_{KG}^{(1)}$ is much larger than $E_{KG}^{(2)}$.
\end{itemize}
\end{theorem}



\begin{remark}
  The $E^{(2)}_{KG}$ electric field is essentially a Klein-Gordon-type
  evolution subjected to a Landau damping external forcing. See
  Section~\ref{sec:Dens} for more details. 
\end{remark}
\begin{remark}
For frequencies bounded away from zero, i.e. for any $\delta > 0$, $P_{\geq \delta }E$, Landau damps at a polynomial rate $\brak{t}^{-\sigma}$ provided the initial data is $W^{\sigma,1}$, hence the extremely slow Landau damping of $E_{KG}$ manifests only at $k=0$. 
As expected, for frequencies bounded from zero, the electric field (both $E_{LD}$ and $E_{KG}$) damps exponentially fast if the initial data is analytic (and analogously $e^{-\brak{kt}^s}$ for Gevrey initial data).   
\end{remark}
\begin{remark}
  As the electric field is not shielded, it is too restrictive to
  assume that $E_{in} \in L^1$, regardless of how well localized the
  initial $h_{in}$ is. If $x h_{in} \in L^1$ and
  $\int_{\R^3_X \times \R^3_v} h_{in} dx d v = 0$, then
  $E_{in}(x) \approx \abs{x}^{-3}$ as $x \to \infty$ (hence,
  $E_{in} \in L^p$ for $p > 1$ but not $p=1$) but one cannot obtain
  faster decay at infinity unless one has higher zero-moment
  conditions, such as
  $\int_{\R^3_x \times \R^3_v} x^\alpha h_{in} d x d v = 0$ and the latter are not propagated by the
    semi-group. 
\end{remark}
\begin{remark}
  In the case of the kinetic free transport, the electric field decays
  only as $\norm{E_{FT}(t)}_{L^\infty_x} \lesssim \brak{t}^{-2}$,
  which is a full power of $t$ slower than \eqref{ineq:EFTbd}.  For
  frequencies bounded away from zero, both the free
    transport electric field $E_{FT}$ and our $E_{LD}$ are
    Landau-damped exponentially fast for analytic data (and polynomial for Sobolev data); the
  difference in decay rates comes from the contribution of long waves.
\end{remark}
\begin{remark}
  There is a clear distinction between Landau damping and dispersive
  decay above: for the Landau damping contributions, each derivative
  buys one power of time (a structure that can be guessed from free
  transport), whereas for the Klein-Gordon-like contributions, there
  is no such behavior.
\end{remark}
\begin{remark}
  We have not endeavored to get the sharpest dependence on the
  initial data that could be possible. This could be important for
  nonlinear extensions.
\end{remark}

We deduce a similar decomposition at the level of the distribution
function:
\begin{theorem} \label{thm:Fdecom} The solution to~\eqref{def:VPElin}
  decomposes as $h = h_{KG} + h_{LD}$ with
  $h_{KG} = \tilde{E}_{KG}(t,x) \cdot \grad_v f^0(v)$ for some
  effective electric field $\tilde{E}_{KG}$, with the following
  estimates:
  \begin{itemize}
  \item for all $m \geq 0$, $r \in [1,\infty]$ and all $p \in [2,\infty)$, ($\tilde{E}_{KG}$ is essentially a phase-shifted version of $E_{KG}^{(1)}$ and satisfies the same estimates) 
  \begin{align}
    \label{ineq:hKGStrich}
    \norm{\brak{v}^mh_{KG}(t)}_{L^2_x L^r_v} & \lesssim_m \norm{h_{in}}_{W^{0,1}_{4}} \\
    \norm{\brak{v}^m h_{KG}(t)}_{L^p_x L^r_v} & \lesssim_{p,m} \brak{t}^{-3 \left(\frac{1}{2} - \frac{1}{p}\right)}\norm{\brak{x} h_{in}}_{W^{0,1}_{4}} \\
    \norm{\brak{v}^m \grad_x h_{KG}(t)}_{L^\infty_x L^r_v} & \lesssim_m \brak{t}^{-3/2}\norm{h_{in}}_{W^{0,1}_{4}}; 
  \end{align}
\item $h_{LD}$ scatters to free transport in all $L^{p}_{x,v}$,
  $p > 6$ (provided one has enough initial regularity). In particular,
  if $\brak{\grad}^\sigma \brak{v}^{m}h_{in}$ is integrable for
  $\sigma,m \in N$ large enough, then for all $p > 6$, there exists an
  $h_\infty \in L^{p}_{x,v}$ such that
  \begin{align*}
    \norm{\brak{v}^m \Big[ h_{LD}(t,x -vt,v) - h_{\infty}(x,v)\Big]}_{L^p_{x,v}} \xrightarrow{t \to +\infty} 0. 
  \end{align*}
\end{itemize}
\end{theorem}

\begin{remark}
  Using the Strichartz estimates~\cite{MR1383431} for the transport
  equation one obtains
  \begin{align*}
    \norm{h_{LD}}_{L^q_t L^p_x L^r_v} \lesssim
    \norm{h_{in}}_{W^{\sigma,1}_{m}} 
  \end{align*}
  for all $(p,q,r,a)$ satisfying
  \begin{align*}
    \frac{1}{r} - \frac{1}{n} < \frac{1}{p} \leq \frac{1}{r} \leq 1,
    \quad  1 \leq \frac{1}{p} + \frac{1}{r}, \quad 
    \frac{2}{q} = n\left(\frac{1}{r} - \frac{1}{p}\right), \quad 
    \frac{1}{a} = \frac{1}{2}\left(\frac{1}{r} + \frac{1}{p} \right), \quad
    a>6.
  \end{align*}
\end{remark}

Theorem \ref{thm:ElecDec} shows that the leading-order term in the
asymptotics of the distribution $h$ factorises between a function of
$(t,x)$ and a fixed function of $v$, which is reminiscent of a
hydrodynamical limit \emph{even though the equations are
  collisionless}. One can actually push this intuition further: a
long-wave re-scaling of the electric field converges weakly (in
negative Sobolev spaces) to the electric field solving the linearized
Euler-Poisson system~\eqref{eq:lin-Euler}. See Subsection~\ref{subsec:hydro} for a proof.
\begin{theorem}
  \label{thm:EPcorr}
  Consider an initial data
  $h_{in}(x,v) = \epsilon^3 \mathcal H_0(\epsilon x,v)$ such that
  $\mathcal H_0$ has zero average and
  $\mathcal H_0 \in W^{5+0,1}_{5}$, and denote
  \begin{align*}
    E_0 = \frac{q^2 n_0}{\epsilon_0 m_e} \grad_x (-\Delta_x)^{-1} \int_{\R^3}\mathcal{H}_0(\cdot, v)
    \, dv.
  \end{align*}
  
  Let $\cE$ be the unique solution in the natural energy space
    $\cE \in L^\infty_t H^1_x$ and
    $\partial_t \cE \in L^\infty_t L^2_x$ to the following Klein-Gordon equation, 
  \begin{align}
    \label{eq:ep-lin}
    \left\{
    \begin{array}{l}
      \displaystyle
      \partial_{t}^2 \cE + \left(\omega_p^2 - \frac{3T}{m_e}
      \Delta\right) \cE =0\\[3mm]
      \displaystyle
      \cE(0,x) = E_0(x) \\[3mm]
      \displaystyle
      \partial_t \cE(0,x) = -n_0 \grad_x \cdot \left(\int v h_{in} dv \right).  
    \end{array}
    \right.
  \end{align}
  Then, for all $s \in (5/2,7/2)$, $0 < \epsilon \ll 1$, for all
  $0 < t < \epsilon^{-N}$, there holds
  \begin{align}
    \label{ineq:Hmsdec}
    \norm{E\left(t,\frac{\cdot}{\epsilon}\right)
    -\cE\left(t,\frac{\cdot}{\epsilon}\right)}_{H^{-s}}
    \lesssim \epsilon^2\left(\epsilon^2 + \epsilon^{s-\frac32-0} \brak{t}
    \right)
    \norm{\mathcal{H}_0}_{W^{2,1}_{5}}. 
  \end{align}
  Moreover, given such initial data with appropriate scaling, both the
  electric fields in the linearized Vlasov-Poisson and the linearized
  Euler-Poisson system are asymptotic to the Bohm-Gross dispersion
  relation at large scale: $E(t,x) = E_+ + E_-$ and
  $\cE(t,x) = \cE_+ + \cE_-$ with the following weak $L^2_x$ limit on
  $t \in [-T,T]$ for any $T$ fixed finite
  \begin{align*}
    & \frac{1}{\epsilon^2} e^{\mp i\omega_p \frac{t}{\epsilon^{2}}}
      E_{\pm}\left(\frac{t}{\epsilon^2},\frac{x}{\epsilon}\right)
      \xrightarrow[L^2_x]{\text{weak}} \frac{1}{2}e^{\pm \frac{3T}{m_e}
      i \Delta t} E_0 \\
    & \frac{1}{\epsilon^2} e^{\mp i\omega_p \frac{t}{\epsilon^2}}
      \cE_{\pm}\left(\frac{t}{\epsilon^2},\frac{x}{\epsilon}\right)
      \xrightarrow[L^2_x]{\text{weak}} \frac{1}{2}e^{\pm \frac{3T}{m_e}
      i \Delta t} \cE_0.
  \end{align*}
\end{theorem}

\subsection{Notation}
\label{subsec:notation}
Denote $\brak{x} = (1 + \abs{x}^2)^{1/2}$ and $\brak{\grad}$ the
Fourier multiplier defined by
\begin{align*}
\widehat{\brak{\grad}f}(\xi) = \brak{\xi} \widehat{f}(\xi); 
\end{align*}
we define other Fourier multipliers in the analogous way. We define
$P_{\leq N}$ to be the Littlewood-Paley low-pass filter in $x$, in
particular, define $\chi \in C_c^\infty(B(0,2))$ with $\chi(x) = 1$
for $\abs{x} \leq 1$, and define
\begin{align*}
\widehat{P_{\leq N} f}(k) = \chi\left(\frac{k}{N}\right) \widehat{f}(k). 
\end{align*}
Denote the weighted Sobolev spaces (note that the weight index refers to the variable $v$)
\begin{align*}
  & \norm{f}_{L^p_{x,v}} := \left( \int_{\R^3 \times \R^3} \abs{f(x,v)}^p
  dx dv \right)^{1/p},
  \quad \norm{\rho}_{L^p_{x}} := \left( \int_{\R^3} \abs{\rho(x)}^p dx
    \right)^{1/p} \\[3mm]
  & \norm{f}_{W^{\sigma,p}_{m}} := \left\|\brak{v}^m
  \brak{\grad_{x,v}}^\sigma f \right\|_{L^p}. 
\end{align*}
For inequalities involving norms $W^{\sigma+0,p}_{m}$ we use the
notation $+0$ if we mean that the estimate holds for $\sigma + \delta$
with $\delta\in(0,1)$ with a constant that depends on $\delta$.  For
$p \in [1,\infty]$ we denote $p' = \frac{p}{p-1}$ the H\"older
conjugate.  Let $f:[0,\infty) \rightarrow \Complex$ satisfy
$e^{-\mu t} f(t) \in L^1$ for some $\mu \in \Real$.  Then for all
complex numbers $\Re z \geq \mu$, we can define the Fourier-Laplace
transform via the (absolutely convergent) integral
\begin{align}
  \label{eq:fourier-laplace}
  \hat{f}(z) := \frac{1}{2\pi} \int_0^\infty e^{-zt} f(t) dt. 
\end{align}
This transform is inverted by integrating the $\tilde{f}$ along a
so-called `Bromwich contour' via the inverse Fourier-Laplace
transform: let $\gamma > \mu$ and define:
\begin{align}
  \check{f}(t) := \int_{\gamma - i\infty}^{\gamma + i \infty} e^{zt} f(z) dz. 
\end{align}

\section{Decomposition of the electric field} \label{sec:Dens}

\subsection{Volterra equation}

The most important property of the linearized Vlasov equations is that one can reduce the problem to a
Volterra equation for the density separately for each spatial
frequency.  We now recall how this is done. Writing a Duhamel
representation along free transport gives
\begin{align*}
  h(t,x,v) = h_{in}(x-vt,v) + \int_0 ^t (\nabla_x W *_x \rho)(t,x-v(t-s))
  \cdot \nabla_ v f^0 (v) ds.
\end{align*}
Taking the Fourier transform in $x$ and integrating in $v$ gets (with
$w_0:= \omega_p^2 n_0^ {-1}$)
\begin{align}
  \label{eq:Volt}
  \hat{\rho}(t,k) 
  = \widehat{h_{in}}(k,kt) - w_0
  \int_0^t (t-\tau) \widehat{f^0}(k(t-\tau)) \hat \rho(\tau,k) d\tau. 
\end{align}
Taking the Fourier-Laplace transform~\eqref{eq:fourier-laplace} in
time for $\Re z$ sufficiently large gives
\begin{align}
  \label{eq:LapVolt}
  \tilde{\rho}(z,k) = H(z,k) + \mathcal L (z,k)\tilde{\rho}(z,k), 
\end{align}
where $H(z,k)$ is the Fourier-Laplace transform of
$t \mapsto \widehat{h_{in}}(k,kt)$ and the \emph{dispersion function}
is 
\begin{align}
  \label{eq:gpenintermed} 
  \mathcal{L}(z,k)
  & :=  - w_0 \int_0^{+\infty} t
    \widehat{f^0}(kt) e^{-z t} dt
    = - \frac{w_0}{|k|^{2}} \int_0^{+\infty}
    e^{-\frac{z}{\abs{k}} s} s \, \widehat{f^0}\left(\hat k s\right) ds, 
\end{align}
with the standard notation $\hat k = k/|k|$. This change of variable
allows for the function $s \mapsto s \widehat{f^0}(\hat ks)$ to have
regularity bounds independently of the size of $k$.  Our
assumption~\eqref{eq:f0} implies that $z \mapsto \cL(z,k)$ is an
entire function.  Moreover, we see that away from $z=0$, $\cL(z,k)$ is
also a smooth function of $k$ (even at $k = 0$ as we shall see in the
expansions below).

\subsection{Asymptotic expansions and lower bounds on the dispersion
  function}
\label{sec:estim-disp-funct}

Solving~\eqref{eq:LapVolt} for $\rho$ works, formally at least, except
where $\mathcal{L}(z,k)$ gets close to one. In the case of
not-so-small spatial frequencies $k$, the treatment is similar to that
of the periodic domain $x \in \mathbb T^3$, and so we merely sketch
the argument here; see~\cite{MouhotVillani11,BMM13} for more details.
\begin{lemma}[Resolvent estimates for non-small frequencies] \label{lem:HiFL} There exists a $\lambda >0$ such that for 
  any $\nu_0 > 0$, $\exists \, \kappa > 0$ (depending on $\nu_0$) such that
  \begin{align}
    \label{ineq:kaplwbd}
    \forall \, |k| > \nu_0, \quad
    \inf_{\Re z > -\lambda \abs{k}} \abs{1 - \cL(z,k) } > \kappa. 
  \end{align}
  Furthermore, the following estimate holds uniformly on the critical
  vertical line of this region
\begin{align}
  \label{ineq:HikHiOmega}
  \forall \, |k| > \nu_0, \ \omega \in \R, \quad
  \abs{\mathcal{L}(\lambda \abs{k} + i\omega,k)} \lesssim_\lambda
  \frac{1}{1+\abs{k}^2+ \omega^2}. 
\end{align}
\end{lemma}
\begin{proof}
  The estimate~\eqref{ineq:kaplwbd} is proved in
  e.g. \cite{MouhotVillani11,BMM13}. To see~\eqref{ineq:HikHiOmega},
  we use integration by parts and the analyticity of $f^0$ to get
  \begin{align*}
    \left| \cL(\lambda \abs{k} + i\omega,k) \right|
    & = \frac{w_0}{|k|^2} \left| \int_0^\infty e^{\lambda s -
      i\frac{\omega}{\abs{k}}s} s\widehat{f^0} \left(\hat k s\right)
      ds \right| \\
    & = \frac{w_0}{|k|^2} \left| \int_0^\infty \frac{1}{\left(\lambda - i
      \frac{\omega}{\abs{k}}\right)^2} \partial_s ^2 \left(e^{\lambda s
      - i\frac{\omega}{\abs{k}}s}\right) s\widehat{f^0}\left(\hat k
      s\right) ds \right| \\
    & \leq \frac{w_0}{\left(\lambda^2|k|^2 +\omega^2 \right)} \left[
      \left| n_0 \right| + \int_0^\infty \left( 2 \left| \partial_s
      \widehat{f^0}\left(\hat{k}s\right)\right|  + s  \left| \partial_s ^2 \widehat{f^0}\left(\hat k
      s\right)\right| \right) ds \right] \\
    & \lesssim_\lambda \frac{1}{1+\abs{k}^2+ \omega^2}, 
\end{align*}
which completes the proof. 
\end{proof} 

Next we turn to the low frequency estimates, which are more
challenging, and contain the long-wave dispersive structure. For
$\delta,\delta'>0$, define the following region in the complex plane:
\begin{align}
  \label{eq:deflambda}
  \Lambda_{\delta,\delta'} := \set{z = \lambda + i\omega \in \Complex
  \ : \ \lambda > -\min\Big[(1-\delta)\abs{\omega},\delta' \abs{k}\Big]}.  
\end{align}
\begin{center}
  \includegraphics[width=0.7\textwidth]{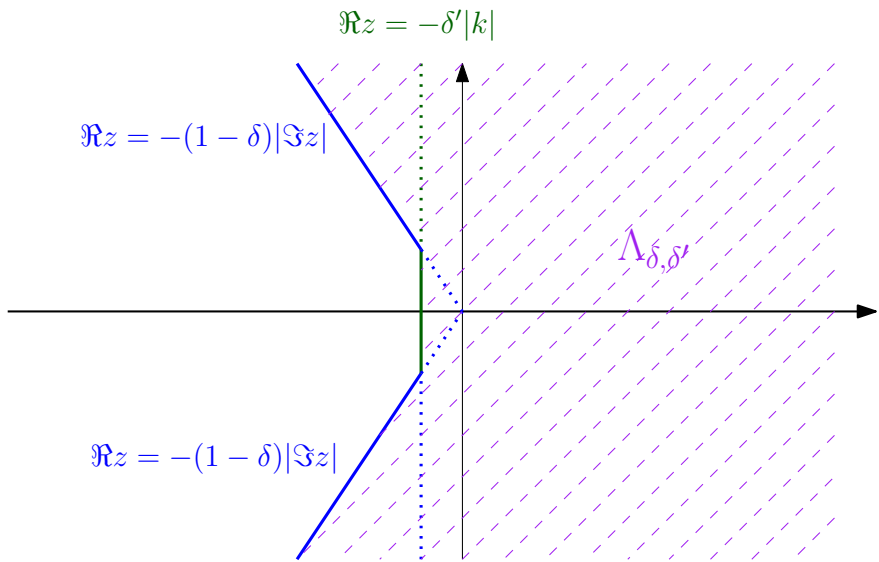}
  \captionof{figure}{The region $\Lambda_{\delta,\delta'}$} \label{pic:lambda} \medskip
\end{center}

We will next show that $\mathcal{L}(z,k)$ stays uniformly away from
one in the region
$\Lambda_{\delta,\delta'} \setminus \set{\abs{z \pm i \omega_p} <
  \epsilon}$ where $\omega_p$ is the cold plasma
frequency. 
The proof relies on two representations: (1) an expansion obtained by
the stationary phase method (i.e. successive integrations by parts in
time) meaningful for large values of $\frac{z}{|k|}$, (2) an
approximation argument using the explicit formula obtained from the
Plemelj formula at the imaginary line $\Re z =0$, that provides
estimates near this line. The first representation is given by the
following lemma.
\begin{lemma}[Asymptotic expansion of $\cL$] \label{lem:Lzexp}
Given $\delta'>0$ sufficiently small depending only on $f^0$,  
\begin{align}
  \label{eq:Lexp}
  \forall \, z \in \Lambda_{\delta,\delta'}, \quad
  \cL(z,k) = -\frac{\omega_p^2}{z^2}\left[1 + \frac{ 3 T\abs{k}^2}{m_ez^2} + \cO\left(\frac{\abs{k}^4}{\abs{z}^4}\right) \right], \quad
  \textup{as} \quad \frac{\abs{z}}{\abs{k}} \to \infty. 
  \end{align}
  Note that this expansion only contains information for frequencies $|k| \ll |z|$.
\end{lemma}
\begin{proof}
  Since $s \mapsto s \hat{f^0}(\hat{k}s)$ is odd, observe that
  $\partial_s ^j (s\widehat{f^0}(\hat{k}s)) |_{s=0} = 0$ for all even
  $j$.  Therefore, integrating by parts repeatedly in $s$ in
 ~\eqref{eq:gpenintermed} gives
  \begin{align*}
    \mathcal{L}(z,k)
    & = -\frac{w_0}{z^2} \widehat{f^0}(0) -
      \frac{w_0}{z^2} \int_0^\infty  e^{-\frac{z}{\abs{k}} s}
      \partial_s ^2 \left(s\widehat{f^0}\left(\hat k s\right)\right) ds
    \\ & = -\frac{w_0}{z^2} \widehat{f^0}(0) -
         \frac{3 w_0 \abs{k}^2}{z^4} \left[ \hat{k}\otimes \hat{k}: \grad^2
         \widehat{f^0}(0)\right] - \frac{w_0 \abs{k}^2}{z^4} \int_0^\infty 
         e^{-\frac{z}{\abs{k}} s} \partial_s ^4 \left(s\widehat{f^0}
         \left(\hat k s\right)\right) ds \\
    & =: -\frac{w_0}{z^2} \widehat{f^0}(0) -
      \frac{3 w_0 \abs{k}^2}{z^4} \left[ \hat{k} \otimes \hat{k} : \grad^2
      \widehat{f^0}(0) \right] - \frac{w_0 \abs{k}^2}{z^4}\zeta(z,k). 
  \end{align*}
  Since $w_0 \widehat{f^0}(0) = w_0 n_0 = \omega_p ^2$ and
  $3 n_0 T := m_e \hat{k} \otimes \hat{k} : \grad^2 \widehat{f^0}(0)$
  gives the leading order terms in~\eqref{eq:Lexp}. Note that we have
  \begin{align}
    \zeta(z,k) := \int_0^\infty e^{-\frac{z}{\abs{k}}s} \partial_s ^4 \left(s \widehat{f^0}\left( \hat{k}s \right) \right) ds. 
  \end{align}
  It remains to show that for $z \in \Lambda_{\delta,\delta'}$, there
  holds
  \begin{align*}
    \abs{\zeta(z,k)} \lesssim \frac{\abs{k}^2}{\abs{z}^2}. 
  \end{align*}
  First consider the region $\Re z \ge -\delta' |k|$: integrating by
  parts two more times, we obtain
  \begin{align*}
    \forall \, \Re z \ge -\delta' |k|, \quad
    \abs{\zeta(z,k)}
    \lesssim \frac{\abs{k}^2}{\abs{z}^2}\left(1 + \int_0^\infty
    e^{\delta' s} \abs{\partial_s^6
    \left(s \widehat{f^0}\left( \hat{k}s \right) \right) } ds \right)
    \lesssim_{\delta'} \frac{\abs{k}^2}{\abs{z}^2}, 
  \end{align*}
  where the last line followed by the analyticity of $f^0$.

  Turn next to the region $\Re z < -\delta' |k|$ with
  $\Re z > - (1-\delta) |\Im z|$. Observe then
  $\arg z^2 \in [\frac{\pi}{2}+\beta,\frac{3\pi}{2}-\beta]$ for a
  small $\beta >0$ depending on $\delta$. Write
  \begin{align*}
    \zeta(z,k)
    = \int_{-\infty}^\infty e^{-\frac{z}{\abs{k}}s} \partial_s ^4
    \left(s \widehat{f^0}( \hat{k}s) \right) ds 
    -\int_{-\infty}^{0} e^{-\frac{z}{\abs{k}}s} \partial_s ^4
    \left(s \widehat{f^0}( \hat{k}s) \right) ds =:
    \zeta_1(z,k) + \zeta_2(z,k).
  \end{align*}
  On the one hand, $\zeta_2$ is bounded as in the region
  $\Re z \ge -\delta' |k|$ due to the now advantageous sign of the
  exponent. On the other hand, $\zeta_1$ is a true Fourier-Laplace
  transform, and due to~\eqref{eq:f0},
\begin{align*}
  \zeta_1(z,k) & = \int_{-\infty}^\infty e^{-\frac{z}{\abs{k}}s} \partial_s ^4
  \left(s \widehat{f^0}( \hat{k}s) \right) ds  =
  \frac{z^4}{|k|^4} \int_{-\infty}^\infty
e^{-\frac{z}{\abs{k}}s} \widehat{f^0}( \hat{k}s) ds  \\
& \qquad = \frac{n_0}{m_e^{3/2}} \frac{z^4}{|k|^4} \int_{-\infty}^\infty
  e^{-\frac{z}{\abs{k}}s -s^2 \frac{T}{2 m_e}} ds
  = \frac{n_0}{m_e^{3/2}} \frac{z^4}{|k|^4} e^{\frac{m_e z^2}{2T\abs{k}^2}}. 
\end{align*}
Due to $\arg z^2 \in [\frac{\pi}{2}+\beta,\frac{3\pi}{2}-\beta]$, it
holds $\Re z^2 \lesssim_{\delta} - |\Im z|^2\lesssim_\delta - |z|^2$ and
this term vanishes (to infinite order) in terms of
$\frac{|z|}{|k|} \to \infty$. This completes the proof.
\end{proof}

Lemma \ref{lem:Lzexp} suffices to estimate the resolvent in much of
the areas of interest.  The next lemma estimates the resolvent for low
frequencies $k$ in the half-plane $\Re z \ge -\delta'|k|$ (assuming
$\delta'$ to be small enough) and away from the cold plasma
frequencies $\pm i \omega_p$. Given $\epsilon >0$, define the
following region:
\begin{align}
  \label{eq:defh}
  \mathbb H_{\epsilon,\delta'} := \set{z = \lambda + i\omega \in \Complex
  : \lambda > -\delta' \abs{k} \ \text{ and } \ |z \pm i \omega_p| \ge
  \epsilon )}.  
\end{align}

\begin{center}
  \includegraphics[width=0.8\textwidth]{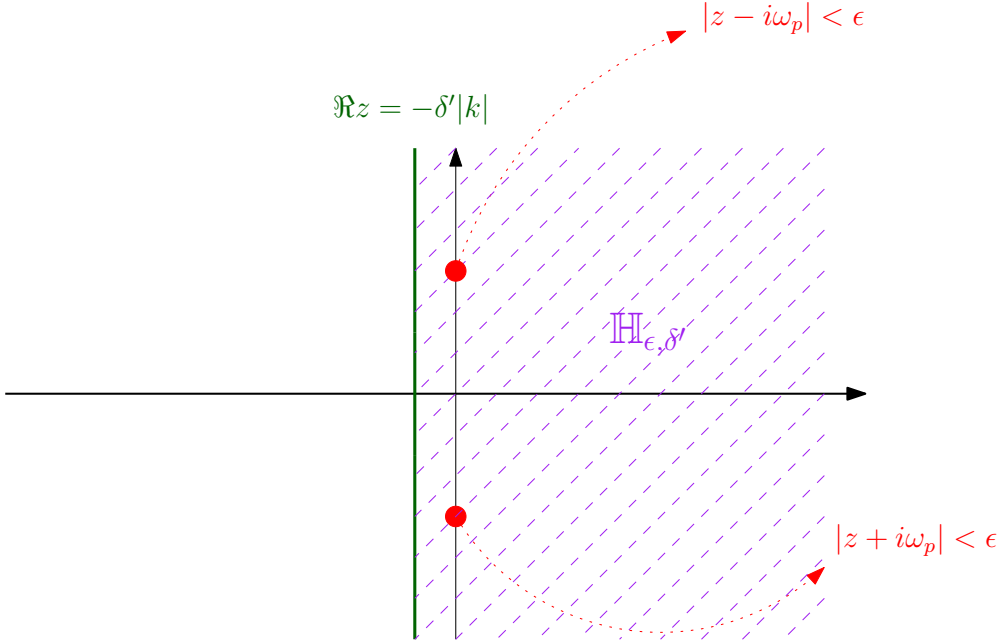}
  \captionof{figure}{The region $\mathbb H_{\epsilon,\delta'}$} \label{pic:hdelta}
  \medskip
\end{center}

\begin{lemma}[Low frequency resolvent estimates]
  \label{theo:resolvent}
  Given $\var,\delta' > 0$, there are
  $\nu_0, \kappa >0$ such that 
  \begin{align}
    \forall \, |k| < \nu_0, \ \forall
    \, z \in \mathbb H_{\epsilon,\delta'}, \quad 
    \abs{1-\mathcal{L}(z,k)} \geq \kappa. 
  \end{align}
\end{lemma}
\begin{proof}
  Let $R > 0$ be fixed large depending on $f^0$ but independent of $k$.
  \smallskip
  
  \noindent
  \textbf{Case 1: $\abs{z} > R \abs{k}$.} In this region the estimate
  follows from~\eqref{eq:Lexp} taking $R > 0$ sufficiently large.
  \smallskip

  \noindent
  \textbf{Case 2: $\abs{z} \leq R \abs{k}$.} In this region the
  asymptotic expansion~\eqref{eq:Lexp} is no longer useful and we use
  the \emph{Plemelj formula} instead.  Writing $z =\lambda + i\omega$,
  it is classical that for $\lambda = 0$, one has (see
  e.g. \cite{Penrose,MouhotVillani11} for explanations),
  \begin{align*}
    \cL(i\omega,k) = \frac{w_0}{|k|^2} \int_\R \frac{(f_k^0)'(r)}{r -
    \frac{\omega}{|k|}} dr + i \frac{w_0 \pi}{|k|^2} (f_k^0)'
    \left(\frac{\omega}{\abs{k}}\right), 
  \end{align*}
  where, for any $k \not =0$, the partial hyperplane average is
  defined as
  \begin{align}
    \forall \, r \in \R, \quad
    f_k^0(r) := \int_{\frac{k}{\abs{k}}r + k_{\perp}} f^0(v_*) dv_*.
  \end{align}
  Moreover, observe that for any $z \in \C$ such that $|z| \le R |k|$
  and $\Re z \in (-\delta'|k|,0]$,
  \begin{align}
    \label{ineq:dzL} 
    \abs{\partial_z \cL(z,k)} \lesssim \frac{w_0}{|k|^3} \int_0
    ^\infty s e^{\delta's} \abs{ \widehat{f^0}\left(\hat{k} s\right)} ds \lesssim
    \frac{1}{\abs{k}^3}
  \end{align}
  where we have used the analyticity of $f^0$ and taken $\delta'$
  small enough.
  \smallskip
  
  \noindent
  {\it Subcase 2.1: $|z| \le R |k|$ and $c|k|\le |\omega| \le R |k|$.}
  Given any $c>0$, we deduce from decay, smoothness, radial symmetry, and monotonicity of $f^0$ that
  \begin{align*}
    \inf_{c < \frac{\abs{\omega}}{\abs{k}} < R}  \frac{w_0 \pi}{|k|^2}
   \left| (f_k^0)'
    \left(\frac{\omega}{\abs{k}} \right)\right|  \gtrsim_{c,R}
    \frac{1}{|k|^2},   
  \end{align*}
  and therefore, using~\eqref{ineq:dzL}, $\abs{\Im \cL(z,k)} \gtrsim 1$
  holds uniformly for all $\Re z \in (-\delta' \abs{k},0)$.
  \smallskip

  \noindent
  {\it Subcase 2.2: $|z| \le R |k|$ and $|\omega| < c |k|$.} Observe
  that $\frac{(f_k^0)'(r)}{r}$ is integrable since $(f_k^0)'(0)=0$ and
  $f^0$ smooth, and if we denote
  \begin{align*}
    \int_\R \frac{(f_k^0)'(r)}{r} dr = \mathfrak c_0 \not =0,
  \end{align*}
  we deduce that if $c \le \frac{|\mathfrak c_0|}{2}$ is small enough
  then
  \begin{align*}
    \abs{\int_\R \frac{(f_k^0)'(r)}{r - \frac{\omega}{\abs{k}}} dr -
    \mathfrak c_0}
     = \abs{\frac{\omega}{\abs{k}}\int_\R \frac{(f_k^0)'(r)}{r
    \left(r-\frac{\omega}{\abs{k}}\right)} dr} 
    \lesssim \frac{\abs{\omega}}{\abs{k}} \le \frac{|\mathfrak c_0|}{2}  
  \end{align*}
  and thus
  \begin{align*}
    \abs{1-\cL(i\omega,k)} \gtrsim \frac{1}{\abs{k}^2}. 
  \end{align*}
  Using again~\eqref{ineq:dzL}, for $\delta'$ small enough we deduce
  that for $\Re z \in (-\delta'|k|,0]$, 
  \begin{align*}
    \abs{1-\cL(z,k)} \gtrsim 1. 
  \end{align*}

  These different cases above together prove that $1/|1-\cL|$ is
  bounded from above on the strip $\Re z \in (-\delta'|k|,0]$ and
  outside $B(0,R|k|) \cap \{ \Re z >0 \}$, and since there are no
  poles within the remaining region $B(0,R|k|) \cap \{ \Re z >0\}$,
  the function is holomorphic in this region and the upper bound is
  also valid there by the maximum principle, which completes the
  proof.
\end{proof} 

\subsection{Construction of the branches of poles}
\label{sec:constr-branch-poles}
From Lemma \ref{lem:Lzexp}, we have a pole at $\abs{k} = 0$ at the
cold plasma frequency: $\cL(\pm i\omega_p,0) = 1$. It follows from
Rouch\'e's theorem that if $\abs{k}$ if small enough, exactly two
poles persist in respective neighborhoods of $\pm i \omega_p$: Given
$\epsilon >0$, the two functions $F(z) := 1-\cL(z,0)$ and
$G(z):=\cL(z,k) - \cL(z,0)$ are holomorphic on the set
$\abs{z \mp i\omega_p} \l \epsilon$, and Lemma \ref{lem:Lzexp} implies
$\abs{ 1- \cL(z,0)} \gtrsim \epsilon$ and
$\abs{\cL(z,0) - \cL(z,k)} \lesssim \abs{k}^2$ on
$\abs{z \mp i\omega_p} =\epsilon$. Therefore, $F(z)=1-\cL(z,0)$ and
$F(z)-G(z)=1-\cL(z,k)$ have the same number of poles in
$\abs{z \mp i\omega_p} < \epsilon$ provided that $|k|$ is sufficiently
small relatively to $\epsilon$.

\begin{center}
  \includegraphics[width=0.8\textwidth]{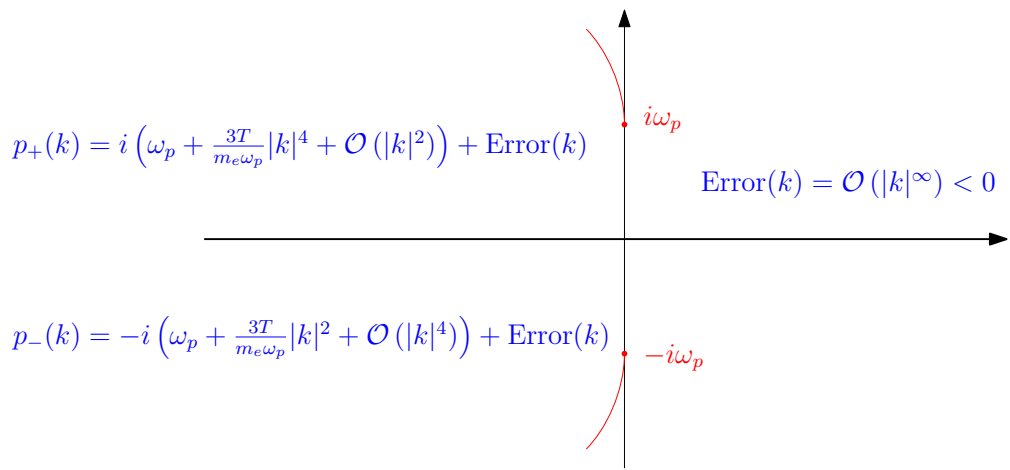}
  \captionof{figure}{The branches of poles $k \mapsto
    p_\pm(k)$}
  \label{pic:branches}
  \medskip
\end{center}
However knowing just the approximate location of the poles is not
enough to deduce dispersive estimates. We next use the implicit
function theorem to construct the branches of solutions $p_{\pm}(k)$.
\begin{lemma} \label{lem:IFT} There are $\epsilon,\nu_0 > 0$ such that
  for all $\abs{k} < \nu_0$, there are unique $p_{\pm}(k) \in \C$
  solution to $\cL(p_{\pm}(k),k) = 1$ in
  $\{|z \mp i\omega_p| < \epsilon\} $ and
  $k \mapsto p_{\pm}(k) =: - \lambda(k) \pm i \Omega(k) $ are smooth
  (but not analytic) and satisfy $\lambda(k) > 0$ and
  $p_{\pm}(k) \sim_{k \to 0} \pm i \omega_p$ with the following
  expansions as $k \to 0$:
   \begin{align}
     \label{eq:OmegaO1}
     & \Omega(k)^2 = \omega_p^2 +  \frac{3T}{m_e}\abs{k}^2 + \cO\left(\abs{k}^4\right)\\
     \label{eq:OmegaO2}
     & \nabla \Omega (k)
       =  i \frac{3T}{m_e \omega_p} k + \cO\left(\abs{k}^3\right) \\
     \label{eq:OmegaO3}
     & \nabla^2\Omega(k)
      = i \frac{3T}{m_e \omega_p} \mbox{{\em Id}}  + \cO\left(\abs{k}^2\right) \\[1mm]
     \label{eq:lamEsts}
     & \abs{\grad^j \lambda (k)}  \lesssim_{j,N} \abs{k}^N \quad
\textup{ for any $j,N \in \N$ }. 
  \end{align}
\end{lemma}

\begin{remark}
  This expansion of $\Im p_\pm(k) = \pm \Omega(k)$ provides the rigorous
  justification for the \emph{Bohm-Gross dispersion} relation in
  kinetic theory. Regarding the real part
  $\Re p_\pm(k) = -\lambda(k)$, physicists assert that
(see e.g. \cite[pp.419]{GoldstonRutherford95}), 
  \begin{align}
    \lambda(k) \approx  \sqrt{\pi}\frac{m_e^{3/2} \omega_p^4}{\abs{k}^3 T^{3/2}}
    \exp \left( - \frac{m_e \omega_p^2}{4\abs{k}^2 T} \right), 
  \end{align}
  however at this time we lack a mathematically rigorous explanation
  for this exact prediction.
\end{remark}
\begin{proof}
  Since $\widehat{f^0}$ is real, $p_+ = \overline{p_-}$ and it is
  enough to build the branch near $+i \omega_p$. By the implicit
  function theorem applied to the function $\cL$ of
  $(z,k) \in \C \times \R^d$, the result follows by verifying
  $\partial_z \cL(i\omega_p,0) \neq 0$, since $\cL$ is smooth in
  $(z,k)$ and analytic in $z$ in this neighborhood. Roughly speaking
  we want to take derivatives of the
  expansion~\eqref{eq:Lexp}. From~\eqref{eq:gpenintermed} and
  integrating by parts as in the proof of Lemma \ref{lem:Lzexp},
  \begin{align*}
    \partial_z \cL(z,k)
    & =  \frac{w_0}{\abs{k}^3}\int_0^\infty
      e^{-\frac{z}{\abs{k}} s} s^2 \widehat{f^0} \left(\hat k s\right) ds  
     = \frac{w_0}{z^2|k|} \int_0^\infty
      e^{-\frac{z}{\abs{k}} s} \partial_s ^2
      \left(s^2 \widehat{f^0}\left(\hat k s\right)\right) ds \\
    & = \frac{2w_0 n_0}{z^3} +
      \frac{w_0}{z^3} \int_0^\infty e^{-\frac{z}{\abs{k}} s}
      \partial_s ^3\left(s^2 \widehat{f^0}
      \left(\hat k s\right)\right) ds \\
    & = \frac{2w_0 n_0}{z^3} +
      \frac{2 w_0|k|}{z^4} \left[ \hat{k} \cdot \nabla
      \widehat{f^0}\left(\hat{k}s\right) \right] + \frac{w_0|k|}{z^4} \int_0^\infty e^{-\frac{z}{\abs{k}} s}
      \partial_s ^4\left(s^2 \widehat{f^0}  \left(\hat k
      s\right)\right) ds \\ 
    & =  \frac{2w_0 n_0}{z^3} +
      \frac{w_0|k|}{z^4} \zeta'(z,k) 
  \end{align*}
  with $\zeta'(z,k)$ that is uniformly bounded for
  $|z - i \omega_p|< \epsilon$ and $|k| < \nu_0$ (using the
  analyticity of $f^0$). Hence
  $\partial_z \cL(i\omega_p,0) = \frac{2 \omega_p^2}{(i\omega_p)^3} =
  \frac{2i}{\omega_p} \not =0$. The implicit function theorem then
  implies the existence of a unique smooth solution $k \mapsto p_+(k)$
  to the equation $\cL(p_+(k),k) = 1$ in a neighborhood of
  $i\omega_p$.

  To get more precise information on the behavior of the poles near
  $k = 0$, we need to compute derivatives in $k$ as well of $\cL$.
  Expansion~\eqref{eq:OmegaO1} immediately follows from Lemmas
  \ref{lem:Lzexp}:
  \begin{align*}
    p_+^2 = p_+^2 \cL(p_+,k)
     =  \omega_p ^2 -
      \frac{3T\omega_p^2}{m_e p_+^2} |k|^2 + \cO \left(|k|^4 \right)  
     =  \omega_p ^2 + \frac{3T}{m_e} |k|^2 + \cO \left(|k|^4\right).
  \end{align*}

  Next, observe that
  \begin{align}
    \label{eq:gradkp}
    \nabla p_+
     = -\frac{(\grad_k \cL)(p_+,k)}{(\partial_z \cL)(p_+,k)},
  \end{align}
  \begin{align}
  \label{eq:grad2kp}
    \nabla^2 p_+
     = - (\partial_z \cL)(p_+,k))^{-1} \left[ (\nabla^2_k
      \cL)(p_+,k) + (\nabla_k \partial_z \cL)(p_+,k) \cdot
      \nabla p_+ + (\partial_z ^2 \cL)(p_+,k) |\nabla p_+|^2  \right]. 
\end{align}
By the same integration by parts method used in Lemma \ref{lem:Lzexp},
we obtain
\begin{align*}
  \nabla_k \cL(z,k)
  & = -\frac{w_0}{\abs{k}^3}  \int_0^\infty e^{-
    \frac{z}{\abs{k}}s}
    s^2 \nabla \widehat{f^0}(\hat{k} s) ds 
   = \frac{2 w_0 k}{z^4} \left[ \hat{k} \otimes \hat{k} : \nabla^2
    \widehat{f^0}(0) \right] + \cO\left(\abs{k}^3\right).
\end{align*}
Using $\nabla^2 \widehat{f^0}(0) = 3n_0 \frac{T}{m_e}
\mbox{Id}$ and $w_0 = \omega_p^2 n_0^{-1}$, this  yields
\begin{align*}
  \nabla_k \cL(p_+(k),k)  = \frac{6 T}{m_e \omega_p^2} k
  + \cO\left(\abs{k}^3\right), 
\end{align*} 
which implies \eqref{eq:OmegaO2}.  The proof of~\eqref{eq:OmegaO3} is
similar, using now~\eqref{eq:grad2kp} instead of~\eqref{eq:gradkp};
the lengthier calculations are omitted for the sake of brevity. In
these calculations a clear pattern emerges: in all derivatives
$\grad_k^m \cL(z,k)$ (respectively $\grad_k^m \partial_z \cL(z,k)$),
for $m \in \N$, the leading order as $k \to 0$ is an even (resp. odd)
power of $z^{-1}$, and thus at $z = \pm i \omega_p$ all derivatives
$\grad_k^m \cL(i\omega_p,0)$ (resp.
$\grad_k^m \partial_z \cL(i\omega_p,0)$) are purely real (resp. purely
imaginary). This proves all derivatives $\nabla^m p_+(0)$ are purely
imaginary and implies thus~\eqref{eq:lamEsts}, i.e. $\Re p_+=-\lambda$
vanishes to infinite order at $k=0$.  Observe that since nevertheless
$\Re p_+ < 0$ as $k \sim 0$ (by Lemma \ref{lem:HiFL}), the function
$k \mapsto p_+(k)$ differs from its Taylor series at $k=0$ and is
therefore not analytic.
\end{proof}

\subsection{Spectral surgery and extraction of Klein-Gordon waves}
\label{sec:surgery}
Through~\eqref{eq:LapVolt}, the solution to the Volterra equation
\eqref{eq:Volt} is classically~\cite{book-volterra} given formally as:
\begin{align}
  \label{eq:VoltSol}
  \hat{\rho}(t,k) = \widehat{h_{in}}(k,kt) +\int_0^t \cR(t-\tau,k)
  \widehat{h_{in}}(k,k\tau) d\tau, 
\end{align} 
where the \emph{resolvent kernel} $\cR$ is given by the inverse
Laplace transform
\begin{align*} 
  \cR(t,k) = \frac{1}{2i\pi}\int_{\gamma - i\infty}^{\gamma + i \infty}
  e^{zt} \frac{\cL(z,k)}{1-\cL(z,k)} dz, 
\end{align*}
for a suitable Bromwich contour such that
$z \mapsto \frac{\cL(z,k)}{1-\cL(z,k)}$ is holomorphic for
$\Re z > \gamma-0$.

The calculations in the two previous
Subsections~\ref{sec:estim-disp-funct}-\ref{sec:constr-branch-poles},
show that for $\abs{k}<\nu_0$ sufficiently small,
$\frac{\cL(\cdot,k)}{1-\cL(\cdot,k)}$ is holomorphic in the region
$\mathbb H_{\epsilon,\delta'}$ represented in Figure~\ref{pic:hdelta}
(the half-plane $\Re z \le - \delta'|k|$ minus $\epsilon$-discs around
for the two poles), with one isolated pole $p_{\pm}(k)$ in each disc,
depending on $k$ as studied in the last Subsection. Therefore, by
Cauchy's Residue theorem,
\begin{align*} 
  \cR(t,k)
  & = \left( \frac{e^{p_+(k) t}}{-\partial_z \cL(p_+(k),k)}
  + \frac{e^{p_-(k) t}}{-\partial_z \cL(p_-(k),k)} \right) 
  + \frac{1}{2i\pi}\int_{\gamma' - i\infty}^{\gamma' + i
    \infty} e^{zt} \frac{\cL(z,k)}{1-\cL(z,k)} dz \\
  & =: \underbrace{\cR_{KG}^+(t,k) + \cR^-_{KG} (t,k)}_{\cR_{KG}(t,k)}
    + \cR_{RFT}(t,k), 
\end{align*}
for some $\gamma'\in (-\delta',-\epsilon)$ so that the vertical line
is to the \emph{left} of the poles $p_{\pm}(k)$ but still in
$\mathbb H_{\epsilon,\delta'}$ where $\cL/(1-\cL)$ is
meromorphic. This decomposes the resolvent $\cR=\cR_{KG} + \cR_{LD}$
into a \emph{Klein-Gordon} part and a \emph{remainder free transport}
part, which then yields a corresponding decomposition of the density
$\hat{\rho}(t,k)$ through~\eqref{eq:VoltSol}:
\begin{align*}
  \hat{\rho}(t,k)
  & = \widehat{h_{in}}(k,kt) +
    \int_0^t \cR_{KG}(t-\tau,k) \widehat{h_{in}}(k,k\tau) d\tau + \int_0^t
    \cR_{RFT}(t-\tau,k) \widehat{h_{in}}(k,k\tau) d\tau \\
  & =:\hat{\rho}_{FT}(t,k) + \hat{\rho}_{KG}^+(t,k)
    + \hat{\rho}_{KG}^-(t,k)  + \hat{\rho}_{RFT}(t,k). 
\end{align*}
We first prove a general expansion of $\hat{\rho}_{KG}^\pm(t,k)$ by
successive integrations in time (producing additional powers of
$k$). Terms in this expansion are either comparable to solutions to
free transport (with additional Fourier multipliers) or to
solutions to a Klein-Gordon-like evolution equation.
\begin{lemma}[Expansion of the Klein-Gordon density]
  \label{lem:RKGexp}
  For all $\abs{k} < \nu_0$ and all $\ell \in \N$, we have
  \begin{align}
    \nonumber
    \hat{\rho}_{KG}^\pm(t,k)
    & = \sum_{j=0} ^\ell e^{p_\pm(k) t} A_j^\pm(k) \left[ k^{\otimes j}
      : \grad_\eta ^j \widehat{h_{in}}(k,0) \right]
      - \sum_{j=0} ^\ell A_j^\pm(k) \left[ k^{\otimes j}
      : \grad_\eta ^j \widehat{h_{in}}(k,kt) \right] \\
    \label{eq:dec-rho-t}
    & \qquad + \int_0 ^t \cR_{KG}(t-\tau,k)
      A^\pm _{\ell+1}(k) \left[ k^{\otimes (\ell+1)}
      : \grad_\eta ^{\ell+1} \widehat{h_{in}}(k,k\tau) \right] d\tau
  \end{align}
  where $\nabla_\eta \widehat{h_{in}}(k,\eta)$ is the differential in
  the second Fourier variable, and with the notation
  \begin{align*}
    A^\pm_j(k) := -\frac{J_\pm(k)}{p_\pm(k)^{j+1}} \quad \text{ and }
    \quad J_{\pm}(k) := - \frac{1}{\partial_z \cL(p_{\pm}(k),k))}. 
  \end{align*}
\end{lemma}
\begin{remark}
Note that the Fourier multipliers $A_j^{\pm}$ are smooth and bounded for $\abs{k} < \nu_0$. 
\end{remark}
\begin{proof}
  Integrating by parts in time gives
  \begin{align*}
    & \hat{\rho}_{KG}^\pm(t,k)
      = \int_0^t \cR_{KG} ^\pm(t-\tau,k)
      \widehat{h_{in}}(k,k\tau) d\tau \\ 
    & = \int_0^t J_\pm(k) e^{p_\pm(k) (t-\tau)} \widehat{h_{in}}(k,k\tau) d\tau
      = -\int_0^t \frac{J_\pm(k)}{p_\pm(k)} \partial_\tau \left(
      e^{-\lambda(k)(t-\tau) \pm i \Omega(k)
      (t-\tau)} \right) \widehat{h_{in}}(k,k\tau) d\tau \\
    & = \frac{J_\pm(k)}{p_\pm(k)} H(k,kt)
      - \frac{J_\pm(k)}{p_\pm(k)} e^{p_\pm(k) t} \widehat{h_{in}}(k,0) 
      + \int_0^t \frac{J_\pm(k)}{p_\pm(k)}
      e^{p_\pm(k) (t-\tau)} \left[ k \cdot \grad_\eta \widehat{h_{in}}(k,k\tau) \right] d\tau, 
  \end{align*}
  and iterating finitely many times yields the result. 
\end{proof}
Note that by symmetry $J_+(k) = \overline{J_-(k)}$ and
$A_j^-(k) = \overline{A_j^+(k)}$, and the calculations of
Lemma~\ref{lem:IFT} give the expansion
$J_\pm(k) = \mp \frac{\omega_p}{2i} + O(\abs{k}^2)$ which allows to
expand the coefficients $A^\pm_j(k)$ in~\eqref{eq:dec-rho-t}. Denoting
$p_\pm(k) = - \lambda(k) \pm i \Omega(k)$ with $\lambda(k) >0$ and
$\Omega(k) = \omega_p + \cO(|k|^2)$, it immediately implies:
\begin{lemma}[Klein-Gordon coefficients at low frequencies]
  \label{lem:PoleExDep}
  One has as $k \to 0$,
  \begin{align}
    \label{def:R0exp}
    A_0 ^+(k) + A_0 ^-(k)
    & = 1 + \cO(\abs{k}^2) \\  
    A_1^+(k) + A_1 ^-(k)
    & = \cO(\abs{k}^2)  \\
    \label{def:Q0exp}
    e^{p_+(k)t} A_0^+(k) + e^{p_-(k)t} A_0^-(k)
    & = e^{-\lambda(k)t} \left[ \cos \left[\Omega(k) t\right] + \cO(\abs{k}^2)
      e^{i\Omega(k) t}
      + \cO(\abs{k}^2) e^{-i\Omega(k) t} \right]  \\
    \label{def:Q1exp}
    e^{p_+(k)t} A_1^+(k) + e^{p_-(k)t} A_1^-(k)
    & = e^{-\lambda(k)t} \left[ \frac{\sin \left[\Omega(k) t\right]}{\Omega(k)}
      + \cO(\abs{k}^2) e^{i\Omega(k) t} + \cO(\abs{k}^2) e^{-i\Omega(k) t} \right], 
  \end{align}
  where the $\cO(\abs{k}^2)$ in the above represent infinitely
  differentiable, bounded functions of $k$ which are independent of
  time.
\end{lemma}
\begin{remark}
  Note crucially that~\eqref{def:R0exp} \emph{cancels} the leading
  order of the free transport evolution for long-waves as $k \to 0$,
  which leads to an improved decay of the Landau damping contribution
  of the electric field for such long-waves.
\end{remark}


We are now able to precise our decomposition of the electric field. 
\begin{definition}\label{def:dec-E}
  We define the variable precision decomposition,
  \begin{align*}
    &\widehat{E}_{LD}^{(1;\ell)}(t,k)
      = w_0\frac{ik}{|k|^2} \left[ 1 - A_0^+(k) - A_0^-(k) \right] \widehat{h_{in}}(k,kt)
      - w_0 \frac{ik}{|k|^2} \sum_{j=1}^\ell A_j^{\pm}(k)
      \left(k^{\otimes j} : \grad_\eta^j \widehat{h_{in}}(k,kt) \right) \\
    &\widehat{E}_{LD}^{(2)}(t,k) =  w_0\frac{ik}{|k|^2} \int_0^t
      \cR_{RFT}(t-\tau,k) \widehat{h_{in}}(k,k\tau) d \tau, \\
    &\widehat{E}_{KG}^{(1;\ell)}(t,k) = w_0\frac{ik}{|k|^2}
      \sum_{j=0}^\ell e^{p_{\pm}(k) t} A_j^{\pm}(k) \left( k^{\otimes
      j} : \grad_\eta ^j \widehat{h_{in}}(k,0) \right) \\
    &\widehat{E}_{KG}^{(2;\ell)}(t,k) = w_0\frac{ik}{\abs{k}^2}
      \int_0 ^t \cR_{KG}(t-\tau,k) A^\pm _{\ell+1}(k) \left[
      k^{\otimes (\ell+1)} : \grad_\eta ^{\ell+1} \widehat{h_{in}}(k,k\tau) \right] d\tau, 
\end{align*}
and accordingly define the particular decomposition we shall use in
the sequel (setting $\ell=4$) 
\begin{align*}
  & E_{LD} ^{(1)}
    := E_{LD}^{(1;4)}, \quad
   E_{LD}^{(2)} := \text{ as above} \\
  & E_{KG}^{(1)} := E_{KG}^{(1;4)}, \quad
   E_{KG}^{(2)} := E_{KG}^{(2;4)}, \\
  & E_{LD} := E_{LD} ^{(1)} + E_{LD} ^{(2)}, \quad
     E_{KG} := E_{KG} ^{(1)} + E_{KG} ^{(2)}.
\end{align*}
\end{definition}


Next, we estimate the `remainder free transport part' of the resolvent.
The gain in powers of $k$ present in \ref{lem:FTRE} is critical to the high quality decay rate of the Landau damping electric field. 
\begin{lemma}[Remainder free transport resolvent at low frequencies]
  \label{lem:FTRE}
  There exists $ \lambda_0 > 0$ such that for all $\abs{k} < \nu_0$ there holds, 
  \begin{align}
    \label{ineq:tildeR}
    \forall \, \abs{k} < \nu_0, \quad
    \abs{\cR_{RFT}(t,k)} \lesssim \abs{k}^3 e^{-\lambda_0 \abs{k} t}. 
  \end{align}
\end{lemma}
\begin{proof}
  We add and subtract by the expected leading order behavior as $k \to
  0$ (using that the integration path is away from $z=0$),
  hence define for $\alpha := 3T\omega_p^2 m_e^{-1}$, 
\begin{align*}
\mathcal{Q}(z) = \frac{\omega_p^2}{z^2 + \omega_p^2} + \frac{ \alpha \abs{k}^2}{(z^2 + \omega_p^2)^2}.  
\end{align*}
The function $z \mapsto e^{zt} \mathcal{Q}(z,k)$ is holomorphic and decaying in the \emph{left} half-plane $\Re z < -\gamma' \abs{k}$,
and hence we deduce, deforming the contour suitably, 
  \begin{align*}
    \cR_{RFT}(t,k) & = \frac{1}{2i\pi } \left( \int_{\Gamma_0} + \int_{\Gamma_+} +
      \int_{\Gamma_-} \right)
      e^{zt} \left(\frac{\cL(z,k)}{1-\cL(z,k)} + \mathcal{Q}(z,k) \right) dz \\ 
    & =: \cR_{RFT}^0 + \cR_{RFT}^+ + \cR_{RFT}^-  \qquad \text{with} \\[3mm]
    & \Gamma_0  = \Big\{z = \lambda + i\omega : \lambda = -\delta\abs{k},
      \omega \in (-R\abs{k}, R\abs{k}) \Big\} \\[3mm]
    & \Gamma_+
      = \Big\{z = \lambda - i(1+\delta)\lambda + i \big[ R-(1+\delta)\delta
      \big] \abs{k} : \lambda \in
      (-\infty,- \delta\abs{k}] \Big\} \\[3mm]
    &  \Gamma_-  = \Big\{z = \lambda + i(1+\delta)\lambda - i\big[R-(1+\delta)\delta \big] \abs{k}: \lambda \in
      (-\infty,- \delta\abs{k}] \Big\}.
  \end{align*}
\begin{center}
  \includegraphics[width=0.7\textwidth]{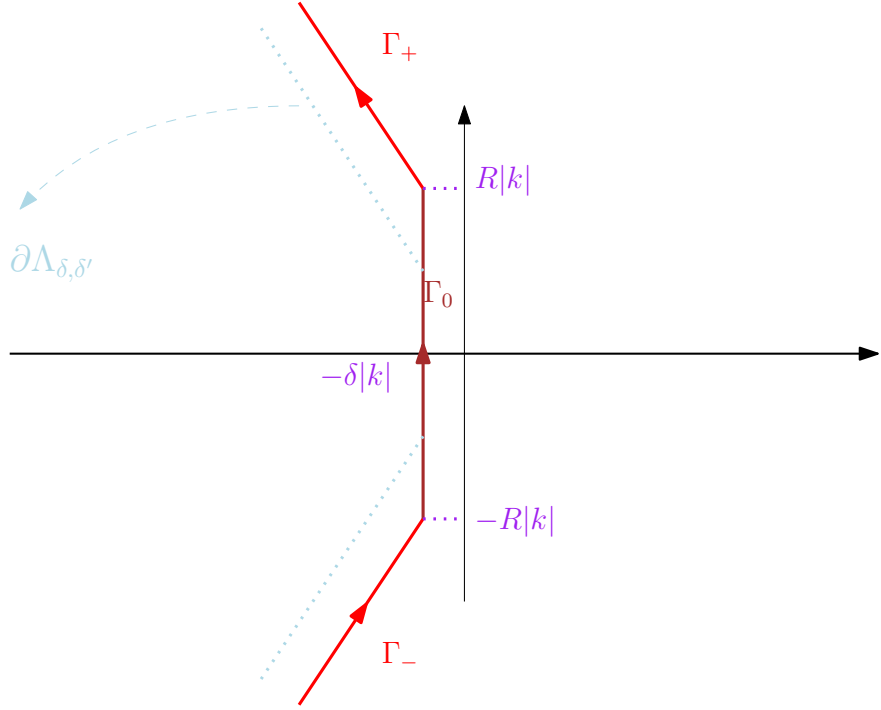}
  \captionof{figure}{The contour of integration (note that it is slightly modified as compared to $\partial \Lambda_{\delta,\delta'}$)}
  \label{pic:contour}
  \medskip
\end{center}
We separate cases as in Lemma \ref{theo:resolvent}. Consider first
$z \in \Gamma_0$ with $\abs{\Im z} < \delta'\abs{k}$.  As in the proof
of Lemma \ref{theo:resolvent}, in this region, there holds the
following expansion, valid for $\delta'$ sufficiently small,
  \begin{align*}
    \forall \, z \in \Gamma_0 \text{ with } \abs{\Im z} < \delta'\abs{k}, \quad
    \cL(\Im z,k) = -\frac{\omega_p^2}{\abs{k}^2} +\cO\left(\frac{\abs{\Im z}}{\abs{k}^3}\right), 
  \end{align*}
and for all $\abs{z} \lesssim \abs{k}$, 
\begin{align*}
\abs{\cL(z,k)} & \lesssim \abs{k}^{-2}, \quad \abs{\partial_z \cL(z,k)} \lesssim \abs{k}^{-3}. 
\end{align*}
Therefore, for $\delta$ and $\delta'$ sufficiently small, we have $\cL(z,k) \approx \abs{k}^{-2}$ for $z \in \Gamma_0$ with $\abs{\Im z} < \delta' \abs{k}$.
Next consider the case $z \in \Gamma_0$ with $\delta'\abs{k} \leq \abs{\Im z} \leq R \abs{k}$. By similar arguments as above and in Lemma \ref{theo:resolvent}, we have $\abs{1 - \cL(z,k)} \gtrsim \abs{k}^{-2}$ and $\abs{\cL(z,k)} \lesssim \abs{k}^{-2}$.  
Therefore, on  $\Gamma_0$, the integrand $\cO( \abs{k}^{2} )$, resulting in the estimate 
\begin{align*}
\abs{\int_{\Gamma_0} e^{zt} \left(\frac{z^2\cL(z,k) + \omega_p^2}{(1-\cL(z,k))(z^2+\omega_p^2)} + \frac{\alpha \abs{k}^2}{(z^2 + \omega_p^2)^2} \right) dz}  \lesssim \abs{k}^3 e^{-\delta \abs{k} t}. 
\end{align*}
This completes the estimates on $\Gamma_0$.

We next turn to $\Gamma_+,\Gamma_-$. We need only consider $\Gamma_+$, $\Gamma_-$ is
analogous. We use now $|z| \gg |k|$ and the decomposition from
Lemma~\ref{lem:Lzexp}:
  \begin{align*}
    \cL(z,k) = -\frac{\omega_p^2}{z^2} -
    \frac{ 3 \omega_p^2 T\abs{k}^2}{m_ez^4}
    + \frac{\abs{k}^4}{z^6} \zeta(z,k)
  \end{align*}
  with $\zeta(z,k)$ uniformly bounded on $\Gamma_+$ and decaying at
  infinity.
  Using this expansion we have 
  \begin{align*}
    \frac{\cL(z,k)}{1-\cL(z,k)} + \mathcal{Q}(z,k) & = \frac{z^2 \cL(z,k) + \omega_p^2}{(1-\cL)(z^2 + \omega_p^2} + \frac{\alpha \abs{k}^2}{(z^2 - \omega_p^2)(z^2 + \omega_p^2)} \\
    & = \frac{\alpha k^2(z^2 \cL - \omega_p^2}{(z^2 - z^2 \cL)(z^2 + \omega_p^2)^2} + \frac{\zeta(z,k) k^4}{z^2 (z^2 - z^2\cL)(z^2 + \omega_p^2)} \\ 
 & = \frac{\alpha k^2( \alpha k^2 + \zeta(z,k) k^4z^{-2}}{(z^2 - z^2 \cL)(z^2 + \omega_p^2)^2} + \frac{\zeta(z,k) k^4}{z^2 (z^2 - z^2\cL)(z^2 + \omega_p^2)}. 
  \end{align*}
Using the uniform boundedness of $\zeta(z,k)$ we integration then gives  
  \begin{align*}
    \abs{\cR_{RFT}^{+}}
     \lesssim e^{-\delta\abs{k} t} \int_{R\abs{k}}^\infty \frac{\abs{k}^4}{x^2} dx = \abs{k}^3 e^{-\delta \abs{k} t}.
\end{align*} 
This completes the proof of Lemma \ref{lem:FTRE}.
\end{proof}

We finally estimate the whole resolvent at frequencies bounded away
from zero, which is simpler.
\begin{lemma}[Non-small frequencies resolvent estimate]
  \label{lem:resolvent-high}
  Given any $\nu_0>0$ there is $\lambda_1 > 0$ such that
  \begin{align}
    \label{ineq:Rest}
    \forall \, |k| \ge \frac{\nu_0}{2}, \quad
    \cR(t,k) \lesssim  \frac{1}{\abs{k}}e^{-\lambda_1 \abs{k} t}.  
  \end{align}
\end{lemma}

\begin{proof}[Proof of Lemma~\ref{lem:resolvent-high}]
  Choose $\lambda >0$ as in Lemma~\ref{lem:HiFL} and deform the
  contour to get (there are no poles in $\Re z \ge - \lambda |k|$ when
  $|k| \ge \nu_0/2$)
  \begin{align*}
    \cR(t,k) = \frac{1}{2i\pi}\int_{-\lambda \abs{k} - i\infty}^{-\lambda
    \abs{k} + i\infty} e^{z t} \frac{\cL(z,k)}{1 - \cL(z,k)} dz
  \end{align*}
  and using the estimate of Lemma~\ref{lem:HiFL} gives 
  \begin{align*}
    \abs{\cR(t,k)}
     \lesssim e^{-\lambda \abs{k} t} \int_{-\infty}^\infty
      \frac{1}{\abs{k}^2 + \abs{\omega}^2} d\omega 
     \lesssim \frac{1}{\abs{k}} e^{-\lambda \abs{k} t}.
  \end{align*}
\end{proof}


\subsection{Another look at the long wave hydrodynamic behavior}
\label{subsec:hydro}

In this subsection we prove Theorem \ref{thm:EPcorr} on the basis of
the decomposition of Definition~\ref{def:dec-E} and the previous
estimates of this section.

\begin{proof}[\textbf{Proof of Theorem \ref{thm:EPcorr}}]
  We prove the first part of the theorem, as the second part is seen directly from the decomposition \ref{def:dec-E}, together with some basic estimates that are similar or easier than what is required to prove the first part.
  For any field $F$, define
  \begin{align*}
    F_\epsilon(t,x) := \frac{1}{\epsilon^3} F
    \left(t,\frac{x}{\epsilon} \right),
    \qquad \text{with Fourier transform } \quad 
    \widehat{F_{\epsilon}}(t,k) = \widehat{F}(t,\epsilon k). 
  \end{align*}
Note that we have defined $h_{in}$ such that for all $\epsilon > 0$, $h_{in,\epsilon}(x,v) = \mathcal{H}_0(x,v)$. 
  Define the initial macroscopic flux (recall that $n_0$ is the total
  mass of $f^0(v)$)
  \begin{align*}
    j_{in}(x) := \frac{1}{n_0} \int_{\R^3} v h_{in} dv \qquad
    \text{with Fourier transform } \quad 
    \widehat{j_{in}}(k) = \frac{1}{n_0} i  \nabla_\eta \widehat{h_{in}}(k,0).
  \end{align*}
Denote,
\begin{align*}
\widehat{J}(k) & = \frac{1}{n_0} i \grad_\eta \widehat{\mathcal{H}_0}(k). 
\end{align*}
  The rescaled electric field $\cE_\epsilon$ in the linearized
  Euler-Poisson system~\eqref{eq:ep-lin} satisfies
  \begin{align*}
\epsilon \widehat{\mathcal{E}}_{\epsilon}(t,k) = \frac{ik}{\abs{k}^2}\widehat{\mathcal{H}_0}(k) \cos \Omega_{KG}(\epsilon k) t
    + \epsilon \omega_p^2 \frac{k}{\abs{k}^2} \left( k\cdot\widehat{J}(k) \right) \frac{\sin \Omega_{KG}(\epsilon k) t}{\Omega_{KG}(\epsilon k)}, 
  \end{align*}
  where we define the exact Euler-Poisson imaginary phase $\Omega_{KG}(k)$
  \begin{align*}
    \Omega_{KG}(k) := \sqrt{\omega_p^2 + \frac{3T}{m} \abs{k}^2}.
  \end{align*}
  Consider $E_{KG}^{(1)}$ first defined in the decomposition of
  Definition~\ref{def:dec-E}, since it is the contribution which is
  asymptotic to $\mathcal{E}_\epsilon$. Lemma~\ref{lem:PoleExDep}
  shows that
  \begin{align*}
    \epsilon\widehat{E}_{KG,\epsilon}^{(1)}(t,k) & = e^{-\lambda(\epsilon k) t} \frac{ik}{\abs{k}^2}\widehat{\mathcal{H}_0}(k) \cos \left[ \Omega(\epsilon k) t \right] + e^{-\lambda(\epsilon k) t} \epsilon \omega_p^2\frac{ik}{\abs{k}^2} \left( k \cdot \widehat{J}( k) \right) 
    \frac{\sin \left[ \Omega(\epsilon k) t\right]}{\Omega(\epsilon k)} \\  
    & \quad + \cO(\epsilon \abs{\epsilon k}^2) e^{-\lambda(\epsilon k) t} \norm{\mathcal{H}_0}_{W^{0,1}_{4}}.
  \end{align*}
  Note that the errors depend on time, but in a uniformly bounded way,
  and that they depend on $v$-moments of $h_{in}$ up to order $5$.
  By the expansion of $\Omega$ in Lemma \ref{lem:IFT}, we
  have
  \begin{align*}
    \left| \Omega(\epsilon k) - \Omega_{KG} (\epsilon k) \right|
    \lesssim \left| \sqrt{\omega_p^2 + \frac{3T}{m} \abs{\epsilon k}^2 +
    \cO\left(|\epsilon k|^4\right)} -
    \sqrt{\omega_p^2 + \frac{3T}{m} \abs{\epsilon k}^2} \right| \lesssim  | \epsilon k|^4
  \end{align*}
  and thus 
  \begin{align*}
    \Big| \cos \left[ \Omega(\epsilon k) t \right] - \cos
    \left[ \Omega_{KG}(\epsilon k) t \right] \Big| =
    \cO \left( \abs{\epsilon k}^4 t \right) , \quad 
    \Big| \sin \left[ \Omega(\epsilon k) t \right] - \sin \left[
    \Omega_{KG}(\epsilon k) t \right] \Big|  = \cO \left(
    \abs{\epsilon k}^4 t \right),
  \end{align*}
  and any $N > 0$, $\abs{1 - e^{-\lambda(\epsilon k) t}} = \cO \left(\abs{\epsilon k}^N t \right)$, and therefore  we deduce,  
  \begin{align*}
    \epsilon \abs{\widehat{E}_{KG,\epsilon}^{(1)} - \widehat{\cE}_\epsilon} \lesssim \left(\abs{\epsilon k}^4 t + \epsilon \abs{\epsilon k}^2\right) \norm{\mathcal{H}_0}_{W^{0,1}_{4}}. 
  \end{align*}
  and thus 
  \begin{align*}
    \epsilon \norm{E_{KG,\epsilon}^{(1)}(t) - \cE_\epsilon(t)}_{H^{-s}}
    & \lesssim \norm{\mathcal{H}_{0}}_{W^{0,1}_{4}}
      \left( \int_{\abs{\epsilon k}\lesssim 1} \frac{ \epsilon^2\abs{\epsilon k}^4 + \abs{\epsilon k}^8 t^2}{\brak{k}^{2s}} dk \right)^{\frac12}
      \lesssim
      \norm{\mathcal{H}_0}_{W^{0,1}_{4}}  \epsilon^{s-\frac{3}{2}-0} \langle
      t \rangle
  \end{align*}
  for $s \in (\frac32,\frac72)$.  Turn next to
  $\widehat{E}_{KG}^{(2)}$; from its definition
  \begin{align*}
    \epsilon \norm{E_{KG,\epsilon}^{(2)}(t)}_{H^{-s}}
    \lesssim \epsilon\int_0 ^t  \left(\int_{\abs{\epsilon k} \lesssim 1}
    \frac{1}{\abs{\epsilon k}^2 \brak{k}^{2s}} \left| \epsilon k \right|^{10}
    \left| \grad_\eta ^{5} \widehat{\mathcal{H}_0} ( k, k
    \tau) \right|^2 dk  \right)^{1/2} d\tau \lesssim \norm{\mathcal{H}_0}_{W^{0,1}_{4}} \epsilon^{s-\frac12-0}, 
  \end{align*}
  for $s \in (\frac32,\frac{19}{2})$. This completes the treatment of
  the `Klein-Gordon parts' of the electric field.

  We turn next to the Landau damping contributions, which in fact
  dominate the error.  It is convenient to subdivide the Landau
  damping field as in Definition \ref{def:dec-E}. The contribution of
  $E_{LD}^{(1)}$ is straightforward, indeed,
  \begin{align*}
    \epsilon\norm{E_{LD,\epsilon}^{(1)}(t)}_{H^{-s}}
    \lesssim \epsilon\left( \sum_{j=0}^4 \int_{\R^3}  \brak{k}^{-2s} \abs{\epsilon k}^{2} \abs{\abs{\epsilon k}^j \widehat{\grad_\eta^j \mathcal{H}}_0(k,kt)}^2  dk \right)^{\frac12} \lesssim
    \norm{\mathcal{H}_0 }_{W^{2,1}_{4}} \frac{\epsilon^2}{\brak{t}}, 
  \end{align*}
  for $s > \frac52$.  Turn finally to $E_{LD}^{(2)}$, which produces
  the dominant error.  Then,
\begin{align*}
  \epsilon \norm{E_{LD,\epsilon}^{(2)}(t)}_{H^{-s}}
  & \lesssim \epsilon\int_0^t \left( \int_{\R^3}
    \frac{\min(\abs{\epsilon k}^{2}, \abs{\epsilon k}^{-1})}{\brak{k}^{2s}}
    e^{-\lambda_0\abs{\epsilon k}(t-\tau)}\abs{\widehat{\mathcal{H}}_0(k,k\tau)}^2 dk  \right)^{1/2} d \tau \lesssim \epsilon^2 \norm{\mathcal{H}_0}_{W^{2,1}_{0}},  
\end{align*}
for $s > \frac{5}{2}$.  This completes the proof of Theorem
\ref{thm:EPcorr}.
\end{proof} 


\section{Electric field estimates}

\subsection{Landau damping estimates on the electric field}
\label{sec:FTest}

In this section we provide estimates for $E_{LD}$. We start with the
optimal decay estimates for the density for the kinetic free transport
(optimal in terms of time decay, not in the dependence on the initial
data). Denote the spatial density of the solution to the free
transport equation
\begin{align*}
  \mathfrak H(t,x) := \int_{\R^3} h_{in}(x-tv,v) dv \quad \text{ with
  Fourier transform } \quad \widehat{\mathfrak H}(t,k) =
  \widehat{h_{in}}(k,kt). 
\end{align*}

\begin{lemma} \label{lem:HbasBd}
  For all $\sigma \geq 0$,
  \begin{align}
    \label{ineq:HL1}
    \norm{ \brak{\grad_x, t \grad_x}^\sigma \mathfrak
    H(t,\cdot)}_{L^1_x} \\ 
    & \lesssim \norm{h_{in}}_{W^{\sigma,1}_{0}}  \\
    \label{ineq:HL2}
    \norm{ \brak{\grad_x, t \grad_x}^\sigma \mathfrak
    H(t,\cdot)}_{L^2_x} \\ 
    & \lesssim \brak{t}^{-3/2}\norm{h_{in}}_{W^{\sigma + 3/2 + 0,1}_{0}}  \\ 
    \label{ineq:HLinf}
    \norm{ \brak{\grad_x,t\grad_x}^\sigma \mathfrak H(t,\cdot)}_{L^\infty_x}
    & \lesssim \brak{t}^{-3}\norm{h_{in}}_{W^{\sigma+ 3+0,1}_{0}}. 
  \end{align}
  More generally, for all $1 \leq p \leq \infty$, $j \geq 0$,
  \begin{align}
    \label{ineq:LpGen}
    \norm{ \brak{\grad_x, t \grad_x}^\sigma \left( \grad_x^{\otimes j} :
    \int_{\R^3} v^{\otimes j} h_{in}(\cdot - tv,v) dv \right)}_{L^p_x}
    & \lesssim \brak{t}^{-j - 3\left(1-\frac{1}{p}\right)}  \norm{h_{in}}_{W^{\sigma+3+j+0,1}_{j}}. 
  \end{align}
\end{lemma}
\begin{proof}
  The inequality~\eqref{ineq:HL1} is clear. To see \eqref{ineq:HLinf}
  note that
  \begin{align*}
    \norm{\grad_x^j \mathfrak H(t,\cdot)}_{L^\infty_x}  \lesssim
    \int_{\R^3} \abs{k}^j\abs{\widehat{\mathfrak H}(t,k)} dk \lesssim
    \left( \int_{\R^3} \frac{\abs{k}^j}{\brak{k,kt}^{3+j+}} dk\right)
    \norm{ h_{in}}_{W^{j+3+,1}_{0}} \lesssim \brak{t}^{-3-\sigma} \norm{ h_{in}}_{W^{j+3+,1}_{0}}. 
  \end{align*}
  The proof of \eqref{ineq:LpGen} follows similarly (using also
  interpolation).
  \end{proof} 
Next, we turn to estimates on the damped part of the electric field.
By Lemmas \ref{lem:PoleExDep} and \ref{lem:HbasBd}, $E^{{(1)}}_{LD}$ satisfies the estimates claimed in Theorem \ref{thm:ElecDec}. 

Turn next to obtaining estimates on $E_{LD}^{(2)}$. 
\begin{lemma}
  There holds the following estimates 
\begin{align}
  \label{eq:FTRE-2}
  \norm{\brak{\grad_x,t\grad_x}^{\sigma} E_{LD}^{(2)}(t)}_{L^2_x}
  & \lesssim \frac{1}{\brak{t}^{5/2}} \norm{h_{in}}_{W^{\sigma+2+0,1}_{0}}  \\ 
  \label{eq:FTRE-infty} 
  \norm{\brak{\grad_x,t\grad_x}^\sigma E_{LD}^{(2)}(t)}_{L^\infty_x}
  & \lesssim \frac{1}{\brak{t}^{4}} \norm{h_{in}}_{W^{\sigma+3+0,1}_{0}}. 
\end{align}
\end{lemma}
\begin{proof}
Consider first the low spatial frequencies $\abs{k} \lesssim \nu_0$.  
Compute using Lemma~\ref{lem:FTRE}, for any $a > 0$, 
\begin{align*}
  & \norm{\brak{\grad_x,t \grad_x}^\sigma E_{LD}^{(2)}(t)}_{L^\infty_x}
   \leq \int_0^t \norm{  \brak{\grad_x,t \grad_x}^\sigma \abs{\grad_x}^{-1}
    R_{RFT}(t-\tau)\ast_x \mathfrak H(\tau,\cdot) }_{L^\infty_x} d\tau \\
  & \qquad \lesssim \int_0^t
    \int_{\R^3} \brak{k, (t-\tau) k}^\sigma \abs{k}^{-1}
    \abs{\cR_{RFT}(t-\tau,k)} \brak{k, \tau k}^{\sigma}
    \abs{\widehat{h_{in}}(k,\tau k)} dk d\tau \\
   & \qquad \lesssim \left[ \int_0^t \abs{k}^2 \left( \int_{\R^3} \langle \tau k
     \rangle^{-3-a} \langle (t-\tau) k \rangle^{-3-a} dk \right)
     d\tau \right] \norm{h_{in}}_{W^{\sigma + 3 + a,1}_{0}}. 
\end{align*}
We split the integral
\begin{align*}
  \int_0^t \left( \int_{\R^3} \abs{k}^2 \langle \tau k \rangle^{-3-a} \langle (t-\tau) k \rangle^{-3-a} dk \right)
     d\tau = \left(\int_0^{\frac{t}{2}} + \int_{\frac{t}{2}} ^t\right) \left(\int_{\R^3} \abs{k}^2 \langle \tau k \rangle^{-3-a} \langle (t-\tau) k \rangle^{-3-a} dk \right) d\tau, 
\end{align*}
and change variables $k'=\tau k$ or $k'=(t-\tau)k$ in each one to obtain 
\begin{align*}
  \int_0^t \left( \int_{\R^3}  \abs{k}^2 \langle \tau k \rangle^{-3-a} \langle (t-\tau) k \rangle^{-3-a} dk \right)
  d\tau \lesssim t^{-4}. 
\end{align*}
Note that the $\abs{k}^2$ in the numerator is crucial for obtaining this sharp rate. 
which concludes the proof of~\eqref{eq:FTRE-infty} at low frequencies.
At frequencies bounded away from zero, the estimate is similar though more straightforward and with less loss of regularity.  
The $L^2$
case~\eqref{eq:FTRE-2} follows similarly and is omitted for brevity.
\end{proof}


\subsection{Dispersive estimates of the electric field} \label{sec:Dispersive}

We now consider the `Klein-Gordon part' of the electric field in
Definition~\ref{def:dec-E}.

Consider first the following useful dispersive estimates:
\begin{lemma}[Dispersive estimates for weakly damped
  poles]
  \label{lem:StrichWkDp}
  We have $1 \leq p \leq 2$ there holds for $t \geq 0$ and
  $\mathfrak f = \mathfrak f(x) \in L^p(\R_x^3)$,
  \begin{align}
    \label{ineq:LpDis}
    \norm{e^{p_{\pm}(\grad)t} P_{\leq \nu_0} \mathfrak f}_{L^{p'}} \lesssim
    t^{-3 \left(\frac{1}{p} - \frac{1}{2} \right) }\norm{\mathfrak f}_{L^p}. 
  \end{align}
\end{lemma}
\begin{remark}
  The linear propagator $e^{p_{\pm}(\grad)t}$ is not a unitary
  operator, and the standard $TT^*$ argument as in e.g. \cite{KT98,TaoTextbook} do not apply.
  Hence, it is not as trivial to obtain the homogeneous Strichartz estimates or the
  full range of expected inhomogeneous Strichartz estimates (although some inhomogeneous Strichartz estimates follow immediately from \eqref{ineq:LpDis} and the O'Neil Young convolution inequality \cite{ON63}).
\end{remark} 
\begin{proof}
The case $p=2$ is immediate.
For the case $p=1$, write
\begin{align}
  e^{p_{\pm}(\grad)t} P_{\leq \nu_0} \mathfrak f = K(t) \ast_x
  \mathfrak f,
\end{align}
with the integral kernel 
\[
  K(t,x) := \frac{1}{(2\pi)^{d/2}} \int_{\R^3} e^{i x \cdot k + i t
    \Omega(k) - \lambda(k) t} a_{\nu_0}(k) dk,
\]
where $a_{\nu_0}$ is a Schwartz class function compactly supported in
a ball of radius $\leq 2\nu_0$ corresponding to the Littlewood-Paley
projection.  Hence the $p=1$ case follows from the pointwise kernel
estimate
\begin{align}
  \label{eq:estim-Kt}
  \norm{K(t)}_{L^\infty} \lesssim t^{-d/2} \qquad (\forall \, t \geq 0). 
\end{align}
The intermediate exponents $p \in (1,2)$ are then obtained by the
Riesz-Thorin interpolation theorem.

Let us prove~\eqref{eq:estim-Kt}. Despite the complex phase,  $K$ is essentially a standard
oscillatory integral, and we may easily adapt the standard arguments as in e.g. [Proposition 6, pg 344 \cite{BigStein}].  Let us first
explain the argument assuming that $\Omega(k) = \abs{k}^2$ (the case
$\Omega(k) = \omega_p^2 + \frac{3T}{m_e}\abs{k}^2$ is the same).  In
this case, we make the change of variables $y = k + \frac{x}{2t}$ and
we write (for some scale $\eps$ chosen below),
\begin{align*}
  K(t,x) & = \frac{e^{i\abs{x}^2/4t^2}}{(2\pi)^{d/2}} \int_{\R^3}
           e^{i t \abs{y}^2 - \lambda t} \chi \left( \frac{y}{\eps}
           \right) a_{\nu_0} \left( y - \frac{x}{2t} \right) dy \\
         & \quad + \frac{e^{i\abs{x}^2/4t^2}}{(2\pi)^{d/2}}
           \int_{\R^3} e^{i t \abs{y}^2 - \lambda t}
           \left[ 1 - \chi\left(\frac{y}{\eps}\right) \right]
           a_{\nu_0} \left( y - \frac{x}{2t} \right) dy \\
         & =: K_S(t,x) + K_{NS}(t,x), 
\end{align*}
where $\chi \in C^\infty_c(B(0,1))$ with $\chi(z) = 1$ for
$\abs{z} \leq 1/2$, and where we have split $K$ into a 'stationary'
part and a 'non-stationary' part. For the stationary part $K_S$ we
simply bound the integrand and estimate the volume of integration:
\begin{align}
\abs{K_S(t,x)} \lesssim_{\nu_0} \eps^{d}.  \label{ineq:Ks}
\end{align}
For the non-stationary part $K_{NS}$ we integrate by parts using the non-vanishing of the phase.
In particular observe that (note that $\abs{\cdot}^2 = x^\ast x$ for $x \in \Complex^n$), 
\begin{align*}
t\abs{2iy - \grad \lambda}^2   e^{it \abs{y}^2 - \lambda t} = \left(-2iy - \grad \lambda\right) \cdot \grad_y e^{it \abs{y}^2 - \lambda t}. 
\end{align*}
Therefore, if we define the differential operator 
\begin{align*}
D \mathfrak{f} := \grad_y \cdot \left( \frac{-2iy - \grad \lambda}{\abs{2iy - \grad \lambda}^2} \mathfrak{f} \right), 
\end{align*}
then by repeated integration by parts we have, 
\begin{align*}
\abs{K_{NS}} \lesssim \frac{1}{t^N}\abs{ \int_{\R^3} e^{i t \abs{y}^2 - \lambda t} D^{N} \left((1 - \chi(\frac{y}{\eps})) a_{\nu_0}( y - \frac{x}{2t})\right) dy }. 
\end{align*}
We next show that the integrand is bounded by $\abs{y}^{-2N}$; for the proof in three dimensions, we only need $N=1,2$ though it holds for all $N$. 
The case $N=1$ is easily checked. Indeed,
\begin{align*}
\partial_j \left( \frac{-2iy - \grad \lambda}{\abs{2iy - \grad \lambda}^2} \right) = -\frac{-2ie_j - \partial_j\grad \lambda}{\abs{2iy - \grad \lambda}^2} = \mathcal{O}(\abs{y}^{-2}), 
\end{align*}
which is sufficient as the terms in which the derivative lands elsewhere are only better (estimates on $\grad^n \lambda$ are provided by Lemma \ref{lem:IFT}). 
For $N=2$ we analogously have 
\begin{align*}
\partial_\ell \left(\left( \frac{-2iy - \grad \lambda}{\abs{2iy - \grad \lambda}^2} \right) \frac{-2ie_j - \partial_j\grad \lambda}{\abs{2iy - \grad \lambda}^2} \right) & = - \left( \frac{-2ie_\ell - \partial_\ell \grad \lambda }{\abs{2iy - \grad \lambda}^2} \right)\left( \frac{-2ie_j - \partial_j\grad \lambda}{\abs{2iy - \grad \lambda}^2} \right)  \\
& \hspace{-3cm} +  \left( \frac{-2iy - \grad \lambda}{\abs{2iy - \grad \lambda}^2} \right)\left( \frac{\partial_{\ell j}\grad \lambda}{\abs{2iy - \grad \lambda}^2} - \frac{
\left(-2ie_\ell - \partial_\ell \grad \lambda\right) \left(2iy-\grad \lambda\right)\cdot \left(-2ie_j - \partial_j \grad \lambda\right) }{\abs{2iy - \grad \lambda}^4}\right) \\
& = \mathcal{O}(\abs{y}^{-4}), 
\end{align*}
which is similarly sufficient (note the pattern that selects a particular dominant term whereas the more complicated error terms are smaller, hence the desired estimates hold for all $N$).
Therefore, provided we choose $N \geq 2$, 
\begin{align*}
\abs{K_{NS}} \lesssim \frac{1}{t^N} \int_{\abs{y} \geq \epsilon} \frac{1}{\abs{y}^{2N}} dy  \lesssim \frac{1}{t^N} \epsilon^{d - 2N}. 
\end{align*}
Hence making the choice $\epsilon \sim t^{-1/2}$ gives the result when combined with \eqref{ineq:Ks}. 
The case using the true $\Omega(k)$ follows by Morse's lemma due to \eqref{eq:OmegaO3} and the other estimates in Lemma \ref{lem:IFT} (see [Proposition 6, pg 344 \cite{BigStein}] for more details).
This completes the main dispersive estimates \eqref{ineq:LpDis}. 

\end{proof}

The estimates on $E_{KG}^{(1)}$ in Theorem \ref{thm:ElecDec} follow immediately from \eqref{ineq:LpDis} and the decomposition \ref{def:dec-E} upon observing that for $\int_{\R^3} \rho_{in} dx = 0$ we have for all $1 < p$, 
\begin{align*}
\norm{\grad_x (-\Delta_x)^{-1} P_{\leq \nu_0} \rho_{in}}_{L^p} \lesssim_{p,\nu_0} \norm{\brak{x}\rho_{in}}_{W^{0,1}_{0}}. 
\end{align*}
Next, we will prove the pointwise-in-time decay estimates on $E_{KG}^{(2)}$.
\begin{lemma} \label{lem:OscError}
  There holds for all $2 \leq p \leq \infty$,
  \begin{align*}
    \norm{ E_{KG}^{(2)} }_{L^p_x} \lesssim \brak{t}^{-3\left(\frac{1}{2} - \frac{1}{p} \right)} \norm{h_{in}}_{W^{3/2,1}_{5}}. 
\end{align*}
\end{lemma}
\begin{proof}
Define $T$ as above:
\begin{align*}
  T(t,x) = \int_{\R^3}  v^{\otimes 5} h_{in}(x-tv,v) dv, \qquad
  \widehat{T}(t,k) = i^5 \grad_\eta^{\otimes 5} \widehat{h_{in}}(k,0). 
\end{align*}
By Minkowski's inequality (and boundedness of the prefactor multipliers on $L^p$ as they are Schwartz class functions), 
\begin{align*}
  \norm{E_{KG}^{(2)}}_{L^p_x}  \lesssim \sum_{+,-}\int_0^t
  \norm{e^{p_{\pm}(\grad_x) (t-\tau)} P_{\leq \nu_0} \abs{\grad_x}^{-1}  \grad_x^{5} : T}_{L^{p}_x} d\tau. 
\end{align*}
On the one hand, by Bernstein's inequality, for any $p \geq r' > 2$,
it follows from \eqref{ineq:LpDis} that we have
\begin{align}
  \norm{e^{p_{\pm}(\grad_x) (t-\tau)} \abs{\grad_x}^{-1} \left( \grad_x^{\otimes 5} : T \right)}_{L^{p}_x}  & \lesssim \norm{e^{p_{\pm}(\grad) (t-\tau)} P_{\leq \nu_0}
    \abs{\grad_x}^{-1} \left( \grad_x^{\otimes 5} : T\right)}_{L^{r'}_x} \\
  & \lesssim \frac{1}{\abs{t-\tau}^{3\left(\frac{1}{r} -
    \frac{1}{2}\right)}} \norm{P_{\leq \nu_0} \abs{\grad_x}^{-1} \left( \grad_x^{\otimes 5} T \right)}_{L^{r}_x}. 
\end{align}
On the the other hand, we similarly have
\begin{align}
  & \norm{e^{p_{\pm}(\grad_x) (t-\tau)} \abs{\grad_x}^{-1} \left(\grad_x^{\otimes 5} : T \right) }_{L^{p}_x} \lesssim \norm{P_{\leq \nu_0} \abs{\grad_x}^{-1} \left(\grad_x^{5} : T \right)}_{L^{2}_x}. 
\end{align}
Therefore, by Lemma \ref{lem:HbasBd}, we have for $r > 1$,
\begin{align*}
  \norm{E_{KG}^{(2)}}_{L^p_x}  \lesssim \norm{h_{in}}_{W^{3/2,1}_{5}} \sum_{+,-}\int_0^t \min\left(\brak{\tau}^{-3/2-3/2},
  \abs{t-\tau}^{-3\left(\frac{1}{r} - \frac{1}{2}\right)} \brak{\tau}^{-3(1-\frac{1}{r}) - \frac{3}{2}} \right) d\tau, 
\end{align*}
which integrates to imply the stated result 
\end{proof} 

\section{Decomposition and scattering for the distribution
  function} \label{sec:DF} In this section we prove Theorem \ref{thm:Fdecom}. Denote the solution
in the free transport moving frame
\begin{align*}
  g(t,x,v) := h(t,x + vt,v), 
\end{align*}
which satisfies $\partial_t g = -E(t,x+tv)\cdot \grad_v
f^0(v)$. Therefore we have on the Fourier side (note
$\widehat{h}(t,k,\eta) = \widehat{g}(t,k,\eta+kt)$),
\begin{align}
  \label{eq:f}
  \widehat{g}(t,k,\eta) = \widehat{h_{in}}(k,\eta)
  - \int_0^t \widehat{E}(\tau,k) \cdot \widehat{\grad_v f^0}(\eta-k\tau) d\tau. 
\end{align}
Consider the contribution $E_{KG}^{(1)}$:
\begin{align*}
  & \int_0^t \widehat{E}_{KG}^{(1)}(\tau,k) \cdot
    \widehat{\grad_v f^0}(\eta-k\tau) d\tau \\
  & =  w_0\frac{ik}{|k|^2} \sum_{j=0}^\ell \sum_{+,-} \int_0^t
    e^{p_{\pm}(k) \tau} A_j^{\pm}(k)
    \left( k^{\otimes j} : \grad_\eta^j \widehat{h_{in}}(k,0) \right) \cdot \widehat{\grad_v f^0}(\eta-k\tau) d\tau \\
  & = w_0\frac{ik}{|k|^2} \sum_{j=0}^\ell \sum_{+,-} e^{p_{\pm}(k) t}
    \frac{A_j^{\pm}(k)}{p_{\pm}(k)}
    \left( k^{\otimes j} : \grad_\eta^j \widehat{h_{in}}(k,0) \right) \cdot \widehat{\grad_v f^0}(\eta-kt)  \\
  & \quad - w_0\frac{ik}{|k|^2} \sum_{j=0}^\ell \sum_{+,-}
    \frac{A_j^{\pm}(k)}{p_{\pm}(k)}
    \left( k^{\otimes j} : \grad_\eta ^j \widehat{h_{in}}(k,0) \right) \cdot \widehat{\grad_v f^0}(\eta)  \\
  & \quad - w_0\frac{ik}{|k|^2} \sum_{j=0}^\ell \int_0^t e^{p_{\pm}(k)
    \tau}\frac{A_j^{\pm}(k)}{p_{\pm}(k)}
    \left( k^{\otimes j} : \grad_\eta ^j \widehat{h_{in}}(k,0) \right)
    \cdot
     \grad_\eta \widehat{\grad_v f^0}(\eta-k\tau) k \,\, d\tau \\
  & =: \widehat{g_{KG}} + \widehat{g_2} + \widehat{g_3}, 
\end{align*}
with $h_{KG},h_2,$ and $h_3$ defined analogously. 
Note that since $\widehat{h_{KG}}(t,k,\eta) = \widehat{g_{KG}}(t,k,\eta + kt)$, 
\begin{align*}
  \widehat{h_{KG}}(t, k,\eta) = \tilde{E}_{KG}(t,k) \cdot \widehat{\grad_v f^0}(\eta), 
\end{align*}
where we define
\begin{align*}
\tilde{E}_{KG}(t,k) := w_0\frac{ik}{|k|^2} \sum_{j=0}^4 \sum_{+,-}
  e^{p_{\pm}(k) t}
  \frac{A_j^{\pm}(k)}{p_{\pm}(k)} \left( k^{\otimes j} : \grad_\eta ^j \widehat{h_{in}}(k,0) \right). 
\end{align*}
Arguing as for the $E_{KG}^{(1)}$ estimates in Theorem \ref{thm:ElecDec}, $h_{KG}$ satisfies \eqref{ineq:hKGStrich}.
Similarly, we define
\begin{align}
  \label{eq:hFT}
  \widehat{h_{LD}}(k,\eta) = \widehat{h_{in}}(k,\eta) + \widehat{g_2}(k,\eta) + \widehat{g_3}(t,k,\eta)  -
  \int_0^t \left( \widehat{E_{KG}}^{(2)} + \widehat{E_{LD}} \right)(\tau,k) \cdot \widehat{\grad_v f^0}(\eta-k\tau) d\tau. 
\end{align}
The term $g_2$ is constant in time and in $L^p$ for all $p \geq 2$ by
the assumptions on the initial data.
The term $g_3$ on the physical side is written in the form
\begin{align}
  h_3(t,x,v) = \int_0^t \grad_x\tilde{E}_3(\tau,x+\tau v) : \left( v \otimes \grad_v f^0(v) \right)  d\tau, 
\end{align}
for a suitable $\tilde{E}_3$. By straightforward variations of the
arguments used to estimate $E_{KG}^{(1)}$ above, we have for any
$6 < p$,
\begin{align*}
  \int_0^t \norm{\grad_x\tilde{E}_3(\tau,x+\tau v) : \left( v \otimes\grad_v
  f^0(v) \right)}_{L^p_{x,v}} d\tau
  & \lesssim \int_0^t \norm{\grad_x\tilde{E}_3(\tau,\cdot)}_{L^p_{x}}
    d\tau \\
  & \lesssim \norm{h_{in}}_{W^{3+0,1}_{5}}
    \int_0^t \brak{\tau}^{-3\left(\frac{1}{2} - \frac{1}{p} \right)} d\tau,  
\end{align*}
and hence $h_3$ converges in $L^p_{x,v}$ for all $p > 6$ as
$t \to \infty$.  Due to the decay estimates in Lemma
\ref{lem:StrichWkDp}, the term in \eqref{eq:hFT} involving
$E_{KG}^{(2)}$ similarly converges in $L^p_{x,v}$ for all $p > 6$.

Next we prove that $h_{LD}$ converges in $L^p_{x,v}$ for $p > 6$.  The
easiest contribution is from $E_{LD}$. Recall the decomposition
$E_{LD} = E_{LD}^{(1)} + E_{LD}^{(2)}$ from Subsection \ref{sec:FTest}. From a straightforward variant of
Lemma \ref{lem:HbasBd}, we have for all $2 \leq p \leq \infty$
\begin{align*}
  \norm{E_{LD}(t)}_{L^p} \lesssim \brak{t}^{-4 + \frac{3}{p}} \norm{h_{in}}_{W^{4+0,1}_{0}}, 
\end{align*}
and hence, for $p \geq 2$,
\begin{align*}
  \int_0^t \norm{ E_{LD} (\tau, x + \tau v) \cdot \grad_v f^0(v)
  }_{L^p_{x,v}} d\tau
  \leq \int_0^t \norm{ E_{LD}(\tau)}_{L^p_x} \norm{\grad_v f^0
  }_{L^p_{v}} d\tau
  \lesssim \norm{h_{in}}_{W^{4+0,1}_{0}}, 
\end{align*}
and so the corresponding contribution converges as $t \to \infty$.



\phantomsection
\addcontentsline{toc}{section}{References}
\bibliographystyle{abbrv}
\bibliography{eulereqns}

\end{document}